\documentclass[11pt]{amsart}

\usepackage{latexsym} 
\usepackage{amsmath} 
\usepackage{epsfig}
\usepackage{amssymb}
\usepackage{enumerate}
\usepackage[dvipsnames]{xcolor}
\usepackage{graphicx}
\usepackage{hyperref}
\usepackage{pdfpages}
\usepackage[normalem]{ulem}

\newtheorem{theorem}{Theorem}[section]

\newtheorem{lemma}[theorem]{Lemma}

\newtheorem{corollary}[theorem]{Corollary}
\newtheorem{proposition}[theorem]{Proposition}

\newcommand{\cE}{{\mathcal E}}

\newcommand{\cM}{{\mathcal M}}
\newcommand{\cN}{{\mathcal N}}
\newcommand{\cP}{{\mathcal P}}
\newcommand{\cS}{{\mathcal S}}
\newcommand{\cT}{{\mathcal T}}

\title[]{When is a set of phylogenetic trees displayed by a normal network?}

\author{Magnus Bordewich, Simone Linz, and Charles Semple} 

\thanks{The second and third authors thank the New Zealand Marsden Fund for their financial support.}

\address{Department of Computer Science, Durham University, United Kingdom}
\email{m.j.r.bordewich@durham.ac.uk}

\address{School of Computer Science, University of Auckland, Auckland, New Zealand}
\email{s.linz@auckland.ac.nz}

\address{School of Mathematics and Statistics, University of Canterbury, Christchurch, New Zealand}
\email{charles.semple@canterbury.ac.nz}


\date{\today}
 
\begin{document}

\begin{abstract}
A normal network is uniquely determined by the set of phylogenetic trees that it displays. Given a set $\cP$ of rooted binary phylogenetic trees, this paper presents a polynomial-time algorithm that reconstructs the unique binary normal network whose set of displayed binary trees is $\cP$, if such a network exists. Additionally, we show that any two rooted phylogenetic trees can be displayed by a normal network and show that this result does not extend to more than two trees. This is in contrast to tree-child networks where it has been previously shown that any collection of rooted phylogenetic trees can be displayed by a tree-child network. Lastly, we introduce a type of cherry-picking sequence that characterises when a collection $\cP$ of rooted phylogenetic trees can be displayed by a normal network and, further, characterise the minimum number of reticulations needed over all normal networks that display $\cP$. We then exploit these sequences to show that, for all $n\ge 3$, there exist two rooted binary phylogenetic trees on $n$ leaves that can be displayed by a tree-child network with a single reticulation, but cannot be displayed by a normal network with less than $n-2$ reticulations.
\end{abstract}

\maketitle

\section{Introduction}

The task of accurately representing the evolutionary history of a set of species can now be approached by using leaf-labelled graphs of different complexities~\cite{blais21}. Traditionally, phylogenetic trees have been inferred to represent speciation events that have given rise to the present-day diversity of species. However, the tree model is, in many cases, too simplistic and does not capture all evolutionary signals in the data. For example, reticulation events such as hybridisation and lateral gene transfer cause patterns of relationships that cannot be represented by a single phylogenetic tree. On the other hand, the reconstruction of more complex leaf-labelled graphs that can capture reticulation events may result in an inferred phylogenetic network that represents information that only has little support from the data. This is, for instance, the case when a phylogenetic network is reconstructed from a set $\cP$ of phylogenetic trees such that the resulting phylogenetic network $\cN$ displays each tree in $\cP$. Typically, $\cN$ also displays additional trees that are not in $\cP$.  Recently, we have characterised those sets of rooted binary phylogenetic trees that can be displayed by a rooted binary level-$1$ network without inferring any additional trees~\cite{doecker24}. Roughly speaking, a level-$1$ network is a rooted phylogenetic network whose underlying cycles do not intersect. As such, level-$1$ networks are a relatively small class of phylogenetic networks that are suited to represent evolutionary relationships with sparse reticulation events.

In this paper, we focus on the reconstruction of normal networks from a collection of rooted phylogenetic trees. Normal networks, which were first introduced by Willson~\cite{willson10}, are phylogenetic networks that have no shortcuts and each non-leaf vertex has a child with in-degree one (formal definitions are given in Section~\ref{sec:prelim}). The class of normal networks has recently been described as a class with prospect for practical use due to their proven mathematical properties~\cite{francis}. Furthermore, it was shown by Francis et al.~\cite{francis21} that each rooted phylogenetic network that is not normal can be transformed into a unique canonical network that is normal by omitting vertices and edges that are, in some sense, redundant.  Referring to the set of all rooted binary phylogenetic $X$-trees that are displayed by a rooted binary phylogenetic network $\cN$ on $X$ as the {\it display set} of $\cN$, the first part of this paper establishes a polynomial-time algorithm that reconstructs a binary normal network whose display set is equal to a given set $\cP$ of rooted binary phylogenetic trees if such a network exists. Importantly, if the algorithm returns a binary normal network, then it is the unique such network because no two binary  normal networks have the same display set~\cite{willson11}. This last result is known to hold for regular networks, which are a superclass of normal networks. Willson~\cite{willson11}  established a top-down algorithm, based on clusters, that always reconstructs a regular network from a collection $\cP$ of rooted phylogenetic trees. In particular, if $\cP$ is the display set of a binary regular network $\cN$, then the algorithm returns $\cN$. Our approach is a bottom-up algorithm based on cherries and reticulated cherries.

As each binary normal network with $k$ vertices of in-degree two has a display set of size $2^k$~\cite{iersel10,willson12}, a simple necessary but not sufficient condition for a collection $\cP$ of rooted binary phylogenetic trees to be the display set of a binary normal network is that $|\cP|=2^{k'}$ for some non-negative integer~$k'$. Hence, most collections of rooted binary phylogenetic trees are not the display set of a binary normal network. In the second part of this paper, we therefore turn to the question of which collections of (arbitrary) rooted phylogenetic trees can be displayed by a (arbitrary) normal network. We show that two rooted  phylogenetic trees can always be displayed by a normal phylogenetic network and, moreover, that the result does not extend to more than two trees. This latter result is in contrast to the class of tree-child networks for which it has been shown that any collection of rooted phylogenetic trees can be displayed by such a network~\cite{linz19}. Like normal networks, tree-child networks~\cite{cardona09} satisfy the structural constraint that each non-leaf vertex has a child of in-degree one but, unlike normal networks, tree-child networks may have shortcuts. 

We then turn to cherry-picking sequences, which are sequences of pairs of leaves of phylogenetic trees. These sequences were introduced by Humphries et al.\ in 2013~\cite{humphries13b} and used to establish the following equivalence: A set $\cP$ of rooted phylogenetic trees  can be displayed by a  temporal tree-child network precisely if  $\cP$ has a cherry-picking sequence. Moreover, by associating a weight to each cherry-picking sequence, these sequences were  also used to characterise the minimum number of reticulations needed over all such networks that display $\cP$.~This optimisation problem was introduced by Baroni et al.~\cite{baroni05} and is know as {\sc Minimum Hybridisation}. Since the publication of~\cite{humphries13b}, cherry-picking sequences have been generalised to larger classes of rooted phylogenetic networks (e.g.,~\cite{janssen21,linz19}). Subsequently, this work has resulted in the development of software to reconstruct phylogenetic networks from collections of trees in practice~\cite{bernardi23,iersel22}. In addition, cherry-picking sequences have recently been used in the context of computing distances between phylogenetic networks~\cite{landry23, landry}. In the present paper, we introduce a new type of cherry-picking sequences to decide if a given collection $\cP$ of phylogenetic trees can be displayed by a normal network. Intuitively, this new type captures the constraint that normal networks do not have any shortcuts. We then exploit the new type of cherry-picking sequences and show that the solution to the {\sc Minimum Hybridisation} problem for $\cP$ over all normal networks is given by the minimum weight of such a sequence for $\cP$. Lastly, we show that, for all $n\geq 3$, there exists a pair of rooted binary phylogenetic trees $\cT$ and $\cT'$ on $n$ leaves such that there exists a binary tree-child network that displays $\cT$ and $\cT'$ with a single reticulation, whereas any binary normal network that displays $\cT$ and $\cT'$ has $n-2$ reticulations, which is the maximum number of reticulations a binary normal network on $n$ leaves can have~\cite{bickner12,mcdiarmid15}.  

The remainder of the paper is organised as follows. Section~\ref{sec:prelim} provides mathematical definitions and concepts that are used throughout the following sections. In Section~\ref{sec:normal-compatible}, we present a polynomial-time algorithm to decide if a given set of rooted binary phylogenetic $X$-trees is the display set of a binary normal network and, if so, to reconstruct such a network. We then show that any two rooted phylogenetic $X$-trees can always be displayed by a normal network in Section~\ref{sec:two-normal} before characterising arbitrarily large collections of rooted phylogenetic $X$-trees that can be displayed by a normal network using cherry-picking sequences in Section~\ref{sec:multi-normal}. In the latter section, we also show how these sequences can be used to solve {\sc Minimum Hybridisation}, provided that the initial collection of rooted phylogenetic $X$-trees can be displayed by a normal network. In Section~\ref{sec:disparity}, we show that, for all $n\ge 3$, there exists a pair of rooted binary phylogenetic $X$-trees on $n$ leaves that can be displayed by a binary tree-child network with a single reticulation and for which any binary normal network that displays the same two trees requires $n-2$ reticulations.

\section{Preliminaries}\label{sec:prelim}

In this section, we introduce notation and terminology that is needed for the upcoming sections. Throughout the paper, $X$ denotes a non-empty finite set. 

\noindent {\bf Phylogenetic networks.} A {\it rooted phylogenetic network $\cN$ on $X$} is a rooted acyclic digraph with no parallel edges or loops that satisfies the following properties:
\begin{enumerate}[(i)]
\item the (unique) root $\rho$ has out-degree two,
\item the set $X$ is the set of vertices of out-degree zero, each of which has in-degree one, and
\item all other vertices either have in-degree one and out-degree two, or in-degree at least two and out-degree one.
\end{enumerate}
For technical reasons, if $|X|=1$, we additionally allow $\cN$ to consist of the single vertex in $X$. The set $X$ is the {\it leaf set} of $\cN$ and the vertices in $X$ are called {\it leaves}. Furthermore, the vertices of in-degree at most one and out-degree two are {\it tree vertices}, while the vertices of in-degree at least two and out-degree one are {\it reticulations}. An edge directed into a reticulation is called a {\it reticulation edge} while each non-reticulation edge is called a {\it tree edge}. We say that $\cN$ is {\it binary} if each reticulation has in-degree exactly two. 

Let $\cN$ be a rooted phylogenetic network on $X$, and let $u$ and $v$ be two vertices of $\cN$. We say that $v$ (resp. $u$) is a {\it descendant} (resp. an {\it ancestor}) of $u$ (resp. $v$) if there is a directed path of length at least zero from $u$ to $v$ in $\cN$.  In particular, if  $(u,v)$ is an edge in $\cN$, then  $u$ is a {\it parent} of $v$ and, equivalently, $v$ is a {\it child} of $u$. Note that we consider a vertex to be an ancestor and a descendant of itself. For an element $x$ in $X$, we also write $p_x$ to denote the (unique) parent of $x$ in $\cN$.  Moreover, the subset of $X$ that precisely contains the  leaves that are descendants of $u$ in $\cN$, denoted by  $C_\cN(u)$ or, simply, $C(u)$ if there is no ambiguity, is called the {\it cluster} of $u$. Hence, for an element $x$ in $X$, we have $C(x)=\{x\}$. Now, let $\{x,y\}$ be a $2$-element subset of $X$. We call $\{x,y\}$  a {\it cherry} of $\cN$ if $p_x=p_y$ and a {\it reticulated cherry} of $\cN$ with {\it reticulation leaf} $x$ if $(p_y,p_x)$ is a reticulation edge in $\cN$. In this paper, we typically distinguish the leaves in a cherry and a reticulated cherry, in which case we write $\{x, y\}$ as the ordered pair $(x, y)$ depending on the roles of $x$ and $y$.

\noindent {\bf Phylogenetic trees.}  A {\it rooted phylogenetic X-tree} $\cT$ is a rooted tree that has no vertex of degree two except possibly the root which has degree at least two. As for rooted phylogenetic networks, $X$ is the {\it leaf set} of $\cT$ and, if $|X| =1$, then $\cT$ consists of the single vertex in $X$. Furthermore, $\cT$ is {\it binary} if $|X| =1$ or, the root  has degree two and all other internal vertices have degree three. For an element $x$ in $X$, we denote by $\cT\backslash x$ the operation of deleting $x$ and its incident edge and, if the parent of $x$ in $\cT$ has out-degree two, suppressing the resulting degree-two vertex. Note that if the parent of $x$ is the root $\rho$ of $\cT$ and $\rho$ has out-degree two, then $\cT\backslash x$ denotes the operation of deleting $x$ and its incident edge, and then deleting $\rho$ and its incident edge. Observe that $\cT\backslash x$ is a phylogenetic $(X-\{x\})$-tree. 

Now, let $\cT$ and $\cT'$ be two rooted phylogenetic $X$-trees. We say that $\cT'$ is a {\it refinement} of $\cT$ if $\cT$ can be obtained from $\cT'$ by contracting a possibly empty set of internal edges in $\cT'$. In addition, $\cT'$ is a {\it binary refinement} of $\cT$ if $\cT'$ is binary.  

\noindent {\bf Naming convention.} Since all phylogenetic trees and networks in this paper are rooted, we omit the adjective {\it rooted} from now on. 

\noindent {\bf Caterpillars and subtrees.}  Let $\cT$ be a  phylogenetic $X$-tree with $|X|\geq 3$. We call $\cT$ a {\em caterpillar} if we can order $X$, say $x_1, x_2, \ldots, x_n$, so that $p_{x_1}=p_{x_2}$ and, for all $i\in \{2, 3, \ldots, n-1\}$, we have that $(p_{x_{i+1}}, p_{x_i})$ is an edge in $\cT$. We denote such a caterpillar $\cT$ by $(x_1, x_2, \ldots, x_n)$ or, equivalently, $(x_2, x_1, x_3, \ldots, x_n)$. Two caterpillars are shown in Figure~\ref{fig:caterpillars}. Furthermore, we call a caterpillar on exactly three leaves a {\em triple}. Again, let $\cT$ be a phylogenetic $X$-tree, and let $Y$ be a subset of $X$. The {\it restriction of $\cT$ to $Y$}, denoted by $\cT|Y$, is the  phylogenetic tree obtained from the minimal rooted subtree of $\cT$ that connects all elements in $Y$ by suppressing each vertex  with in-degree one and out-degree one

\begin{figure}[t]
\center
\scalebox{0.92}{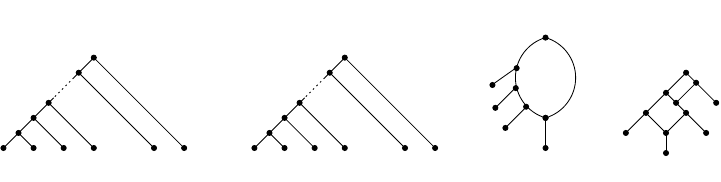}
\caption{The two caterpillars $\cT=(\ell_1,\ell_2,\ell_3,\ldots,\ell_{n-1},\ell_n)$ and $\cT'=(\ell_2,\ell_3,\ell_4,\ldots,\ell_n,\ell_1)$ on $n$ leaves with $n\geq 3$, a tree-child network $\cN$, and a normal network $\cN'$. Each of $\cN$ and $\cN'$ displays $\cT$ and $\cT'$ for $n=4$.}
\label{fig:caterpillars}
\end{figure}

\noindent {\bf Displaying.} Let $\cT$ be a  phylogenetic $X'$-tree, and let $\cN$ be a phylogenetic network on $X$ with $X'\subseteq X$. We say that $\cN$ {\it displays} $\cT$ if, up to suppressing vertices with in-degree one and out-degree one, there exists a binary refinement of $\cT$ that can be obtained from $\cN$ by deleting edges, leaves not in $X'$, and any resulting vertices of out-degree zero, in which case the resulting acyclic digraph is referred to as an {\it embedding} of $\cT$ in $\cN$. More generally, if $\cP$ is a collection of phylogenetic $X'$-trees, then $\cN$ {\it displays} $\cP$ if each tree in $\cP$ is displayed by $\cN$. Moreover, the set of all binary phylogenetic $X$-trees displayed by $\cN$ is referred to as the {\it display set} of $\cN$. Let $\cE$ be an embedding of a phylogenetic $X$-tree in $\cN$. For the purpose of the upcoming sections, we view $\cE$ as a subset of the edge set of $\cN$. Furthermore, for an edge $e$ of $\cN$, we say that $\cE$ {\it uses} $e$ if $e$ is an edge of $\cE$.

\noindent {\bf Tree-child and normal networks.}  Let $\cN$ be a  phylogenetic network on $X$. A reticulation edge $(u,v)$ of $\cN$ is a {\it shortcut} if $\cN$ has a directed path from $u$ to $v$ that avoids $(u, v)$. If a reticulation edge $e=(u,v)$ is a shortcut of $\cN$, then we define the {\it length} of $e$ to be the number of vertices in the union of the vertex sets of all directed paths from $u$ to any parent of $v$.

We say that a phylogenetic network $\cN$ is  {\it tree-child } if each non-leaf vertex in $\cN$ has a child that is a leaf or a tree vertex. A vertex $v$ of $\cN$ is said to have a tree path if there is a directed path from $v$ to some leaf $x$ in $\cN$ such that every vertex in the path, except possibly $v$ itself, is a tree vertex or a leaf. The following equivalence is well known and freely used in this paper.

\begin{lemma}\label{l:tree-path}
Let $\cN$ be a  phylogenetic network on $X$. Then $\cN$ is tree-child if and only if each vertex of $\cN$ has a tree path.
\end{lemma}

\noindent Furthermore, we say that $\cN$ is {\it normal} if it is tree-child and has no shortcut. If $\cN$ is a binary normal network with $k$ reticulations, then $\cN$ displays exactly $2^k$ binary phylogenetic $X$-trees~\cite{iersel10,willson12}. 

Now let $\cP$ be a collection of  phylogenetic $X$-trees. We say that $\cP$ is {\em normal compatible} if there exists a normal network on $X$ that displays $\cP$. To illustrate, Figure~\ref{fig:caterpillars} shows a tree-child network $\cN$ that contains a shortcut and a normal network $\cN'$ that both  display $\cP=\{(\ell_1,\ell_2,\ell_3,\ell_4),(\ell_2,\ell_3,\ell_4,\ell_1)\}$. Hence, $\cP$ is normal compatible. 

\noindent{\bf Hybridisation number.}
With the definitions of a tree-child and normal network in hand, we end this section by introducing three quantities that will play an important role in solving {\sc Minimum Hybridisation} over all normal networks. Let $V$ be the vertex set, and let $\rho$ be the root of a phylogenetic network $\cN$ on $X$. The {\it hybridisation number} of $\cN$, denoted $h(\cN)$, is the value
$$h(\cN)=\sum_{v\in V-\{\rho\}} \left(d^-(v)-1\right),$$
where $d^-(v)$ denotes the in-degree of $v$. Furthermore, for a set $\cP$ of phylogenetic $X$-trees, we set
$$h_{\rm n}(\cP)=\min\{h(\cN):~\mbox{$\cN$ is a normal network on $X$ that displays $\cP$}\} $$
and
$$h_{\rm tc}(\cP)=\min\{h(\cN):~\mbox{$\cN$ is a tree-child network on $X$ that displays $\cP$}\}.$$

\section{Sets of binary phylogenetic trees that are the display set of a binary normal network}\label{sec:normal-compatible}

In this section, we establish a simple recursive algorithm for deciding if a collection $\cP$ of binary phylogenetic trees is the display set of a binary normal network $\cN$. We say that $\cP$ is {\em tightly normal compatible} if there exists a binary normal network on $X$ whose display set is exactly $\cP$. If $\cP$ is tightly normal compatible, then $|\cP|=2^k$ for some non-negative integer $k$, and the binary normal network whose display set is $\cP$ is unique and has exactly $k$ reticulations~\cite{willson11}.

Called {\sc Display Set Compatibility}, we begin with a formal description of the algorithm.

\noindent {\sc Display Set Compatibility} \\
\noindent {\bf Input:} A collection of binary phylogenetic $X$-trees. \\
\noindent {\bf Output:} A binary normal network on $X$ whose display set is $\cP$ or the statement {\em $\cP$ is not tightly normal compatible}.

\begin{enumerate}[{\bf Step 1.}]
\item If there is no non-negative integer $k$ such that $|\cP|=2^k$, then return {\em $\cP$ is not tightly normal compatible}.

\item If $|X|=1$, then return the binary normal network consisting of a single vertex labelled with the element in $X$.

\item If $|X|=2$, then return the binary normal network consisting of a root vertex and two leaves bijectively labelled with the elements in $X$.

\item \label{all-step} If there is subset $\{a, b\}$ of $X$ such that $\{a, b\}$ is a cherry of every tree in $\cP$, then
\begin{enumerate}[{\bf (a)}]
\item Delete $b$ from each tree in $\cP$, and set $\cP'$ to be the resulting collection of binary phylogenetic $(X-\{b\})$-trees.

\item Apply {\sc Display Set Compatibility} to $\cP'$.
\begin{enumerate}[{\bf (i)}]
\item If a binary normal network $\cN'$ on $X-\{b\}$ is returned, construct a binary normal network $\cN$ on $X$ by subdividing the edge directed into $a$ with a new vertex $p_a$ and adjoining a new leaf $b$ to $p_a$ via a new edge $(p_a, b)$. Return $\cN$.

\item Else, return {\em $\cP$ is not tightly normal compatible}.
\end{enumerate}
\end{enumerate}

\item \label{half-step} If there is a subset $\{a, b\}$ of $X$ such that $\{a, b\}$ is a cherry of exactly half of the trees in $\cP$ and, for some $x\in \{a, b\}$, there is a bijection $\varphi$ from the subset $\cP_1$ of trees in $\cP$ with $\{a, b\}$ as a cherry to $\cP-\cP_1$ such that, for all $\cT_1$ in $\cP_1$, the pair $\{a,b\}$ is a reticulated cherry with reticulation leaf $x$ of the unique binary normal network whose display set is $\{\cT_1, \varphi(\cT_1)\}$, then
\begin{enumerate}[{\bf (a)}]
\item Apply {\sc Display Set Compatibility} to $\cP-\cP_1$.
\begin{enumerate}[{\bf (i)}]
\item If a binary normal network $\cN'$ on $X$ is returned, construct a binary phylogenetic network $\cN$ on $X$ by subdividing the edges directed into $x$ and $y$ with new vertices $p_x$ and $p_y$, respectively, and adding the new edge $(p_y,p_x)$, where $y\in\{a,b\}-\{x\}$.

\item If $\cN$ is normal, return $\cN$.

\item Else, return {\em $\cP$ is not tightly normal compatible}.
\end{enumerate}
\end{enumerate}

\item Else, return {\em $\cP$ is not tightly normal compatible}.
\end{enumerate}

\noindent Before continuing, we make some remarks concerning {\sc Display Set Compatibility}. If $\cN$ is a binary normal network, then either $\cN$ has a cherry or a reticulated cherry~\cite{bor16}. Suppose that $\cP$ is the display set of $\cN$.  If $\cN$ has a cherry $\{a, b\}$, then every tree in $\cP$ has $\{a, b\}$ as a cherry. Step~\ref{all-step} is checking for this property. Analogously, if $\cN$ has a reticulated cherry $\{a, b\}$ with reticulation leaf $b$, then $\cP$ satisfies the conditions in Step~\ref{half-step}. The algorithm {\sc Display Set Compatibility} recursively checks for these conditions in Steps~\ref{all-step} and~\ref{half-step} and, if satisfied, reduces the size of $|X|$ or the size of $|\cP|$. 

To illustrate the working of the algorithms, applying {\sc Display Set Compatibility} to the set $\cP$ of eight binary phylogenetic $X$-trees with $X=\{1,2,3,4,5\}$ as shown in Figure~\ref{fig:tightly-normal-compatible}, the algorithm starts by checking if any 2-element subset of $X$ is a cherry in all trees of $\cP$ in Step~\ref{all-step}. Since this is not the case, Step~\ref{half-step} checks if any such subset is a cherry in half of the trees of $\cP$. Assuming that the algorithm iterates through all  $2$-element subset of $X$ in numerical ordering, $\{1,2\}$ is the first such subset that is found because $\{1,2\}$ is a cherry of each tree in $\cP_1=\{\cT_1, \cT_3, \cT_4,\cT_6\}$. Moreover, there is a bijection $\varphi:\cP_1\to \cP-\cP_1$ such that, for each $\cS\in\cP_1$, $\{1,2\}$ is a reticulated cherry with reticulation leaf 2 of the unique binary normal network whose display set is $\{\cS,\varphi(\cS)\}$. In particular, we have $\varphi(\cT_1)=\cT_5$, $\varphi(\cT_3)=\cT_7$, $\varphi(\cT_4)=\cT_8$, and $\varphi(\cT_6)=\cT_2$.  {\sc Display Set Compatibility} is then recursively called for $\cP-\cP_1$.

To establish the correctness of {\sc Display Set Compatibility}, we begin with three lemmas. Let $\cN$ be a phylogenetic network on $X$, and suppose that $\{x, y\}$ is a reticulated cherry of $\cN$, where $x$ is the reticulation leaf. Observe that $(p_y, p_x)$ is an edge of $\cN$. Furthermore, we denote the parent of $p_x$ that is not $p_y$ by $g_x$. Throughout the remainder of this section, we freely use the fact that if $\cP$ is a collection of binary phylogenetic $X$-trees and $\cP$ is the display set of a binary normal network on $X$, then $\cN$ is the unique such network~\cite[Corollary~3.2]{willson11}. Moreover, if $\cT\in \cP$, then there is a unique embedding of $\cT$ in $\cN$.

\begin{figure}[t]
\center
\scalebox{0.92}{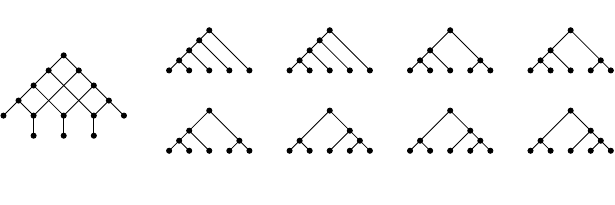}
\caption{Eight binary phylogenetic trees $\cP=\{\cT_1,\ldots,\cT_8\}$ and the unique binary normal phylogenetic network that displays $\cP$.}
\label{fig:tightly-normal-compatible}
\end{figure}

\begin{lemma}
Let $\cN$ be a binary normal network on $X$, and let $\cP$ be the display set of $\cN$.
\begin{enumerate}[{\rm (i)}]
\item If $\{a, b\}$ is a cherry of $\cN$, then $\{a, b\}$ is a cherry of every tree in $\cP$.

\item If $\{a, b\}$ is a reticulated cherry of $\cN$ with $a$ being the reticulation leaf, then $\{a, b\}$ is a cherry of exactly half of the trees in $\cP$. Furthermore, if $\cP_1$ denotes the subset of trees in $\cP$ with $\{a, b\}$ as a cherry, then there is a bijection $\varphi: \cP_1\rightarrow \cP-\cP_1$ such that, for all $\cT_1\in \cP_1$, the pair $\{a,b\}$ is a reticulated cherry with reticulation leaf $a$ of the unique binary normal network whose display set is $\{\cT_1, \varphi(\cT_1)\}$.
\end{enumerate}
\label{cherries}
\end{lemma}

\begin{proof}
The proof of (i) is straightforward and omitted. For the proof of (ii), first observe that half of the trees in $\cP$ use $(p_b, p_a)$ in their (unique) embedding in $\cN$. Let $\cT_1$ be such a tree. Then $\{a, b\}$ is a cherry of $\cT_1$. Let $\cT$ be the tree in $\cP$ whose embedding in $\cN$ is obtained from the embedding of $\cT_1$ in $\cN$ by replacing the edge $(p_b, p_a)$ with $(g_a, p_a)$. Then the  normal network obtained from $\cN$ by deleting all reticulation edges except $(p_b, p_a)$ and $(g_a, p_a)$, and suppressing the resulting degree-two vertices is the unique binary normal network whose display set is $\{\cT_1, \cT\}$. Part (ii) now follows.
\end{proof}

In part, the next two lemmas are converses of Lemma~\ref{cherries}(i) and~(ii), respectively.

\begin{lemma}
Let $\cP$ be a collection of binary phylogenetic $X$-trees. Suppose that $\{a, b\}$ is a cherry of every tree in $\cP$. Let $\cP'$ be the collection of binary phylogenetic $(X-\{a\})$-trees obtained from $\cP$ by replacing each tree $\cT$ in $\cP$ with $\cT\backslash a$. Then $\cP$ is tightly normal compatible if and only if $\cP'$ is tightly normal compatible. Furthermore, if $\cP$ is tightly normal compatible, then $\{a, b\}$ is a cherry of the unique binary normal network on $X$ whose display set is $\cP$.
\label{all-cherries}
\end{lemma}

\begin{proof}
First suppose that $\cP$ is tightly normal compatible. Then there is a unique binary normal network $\cN$ on $X$ whose display set is $\cP$. If $p_a$ is a reticulation in $\cN$, then, as $\cN$ has no shortcuts, there is a tree path from a parent of $p_a$ to a leaf that is not $b$, and so there is a tree in $\cN$ that does not have $\{a, b\}$ as a cherry, a contradiction. Thus $p_a$ is a tree vertex and, similarly, $p_b$ is a tree vertex. Let $u$ be the child vertex of $p_a$ that is not $a$. Since $\cN$ is normal, there is a tree path from $u$ to a leaf $\ell$. If $\ell\neq b$, then there is a tree in $\cP$ whose unique embedding in $\cN$ uses $(p_a, u)$ and does not have $\{a, b\}$ as a cherry, a contradiction. Therefore $\ell=b$. Assume that $p_a\neq p_b$, and let $v$ be the child vertex of $p_b$ that is not $b$. Then there is a tree path in $\cN$ from $v$ to a leaf $\ell'$ such that $\ell'\not\in \{a, b\}$. But then there is a tree in $\cP$ whose unique embedding in $\cN$ uses $(p_b, v)$ and displays the triple $(b,\ell',a)$, a contradiction. Hence $p_a=p_b$ and so $\{a, b\}$ is a cherry of $\cN$. In turn, this implies that, if $\cN'$ is the normal network obtained from $\cN$ by deleting $a$, suppressing the resulting degree-two vertex and, if $p_a$  coincides with the root of $\cN$, additionally deleting the root, then $\cN'$ displays $\cP'$ and is the unique binary normal network whose display set is $\cP'$.

Now suppose that $\cP'$ is tightly normal compatible. Then there is a unique binary normal network $\cN'$ on $X-\{a\}$ whose display set is $\cP'$. Let $\cN$ be the normal network on $X$ obtained from $\cN'$ by subdividing the edge directed into $b$ with a new vertex $p_b$ and adjoining a new leaf $a$ to $p_b$ via a new edge $(p_b, a)$. Since the display set of $\cN'$ is $\cP'$, it is easily seen that the display set of $\cN$ is $\cP$. It follows that $\cP$ is tightly normal compatible.
\end{proof}

\begin{lemma}
Let $\cP$ be a collection of binary phylogenetic $X$-trees. Suppose that $\{a, b\}$ is a cherry of exactly half of the trees in $\cP$, and let $\cP_1$ be the collection of such trees. Furthermore, suppose that, for some $x\in \{a, b\}$, there is a bijection $\varphi: \cP_1\rightarrow \cP-\cP_1$ such that, for all $\cT_1$ in $\cP_1$, the pair $\{a, b\}$ is a reticulated cherry with reticulation leaf $x$ of the unique binary normal network whose display set is $\{\cT_1, \varphi(\cT_1)\}$.
\begin{enumerate}[{\rm (i)}]
\item If $\cP$ is tightly normal compatible, then $\cP-\cP_1$ is tightly normal compatible, and $\{a, b\}$ is a reticulated cherry with reticulation leaf $x$ of the unique binary normal network on $X$ whose display set is $\cP$.

\item Suppose that $\cP-\cP_1$ is tightly normal compatible, and let $\cN'$ be the unique binary normal network on $X$ whose display set is $\cP-\cP_1$. Let $\cN$ be the network on $X$ obtained from $\cN'$ by subdividing the edges directed into $x$ and $y$ with $p_x$ and $p_y$, respectively, and adding the new edge $(p_y,p_x)$, where $\{x,y\}=\{a,b\}$. If $\cN$ is normal, then $\cP$ is tightly normal compatible, and $\cN$ is the unique binary normal network whose display set is $\cP$.
\end{enumerate}
\label{half-cherries}
\end{lemma}

\begin{proof}
In reference to the initial hypothesis of the lemma, we may assume without loss of generality that $x=a$. For the proof of (i), suppose that $\cP$ is tightly normal compatible, and let $\cN$ be the unique binary normal network on $X$ whose display set is $\cP$. We first show that $\{a, b\}$ is a reticulated cherry of $\cN$ with reticulation leaf $a$. Assume that $p_a$ and $p_b$ are both tree vertices in $\cN$. If $p_a=p_b$, then $\{a, b\}$ is cherry of $\cN$, and so, as $\cN$ displays $\cP$, every tree in $\cP$  has $\{a, b\}$ as a cherry, a contradiction. Thus $p_a\neq p_b$. Say that neither $p_a$ is an ancestor of $p_b$ nor $p_b$ is an ancestor of $p_a$ in $\cN$. If either $p_a$ or $p_b$ is a parent of a tree vertex or a leaf, then no tree in $\cP$ has $\{a, b\}$ as a cherry. Furthermore, if each of $p_a$ and $p_b$ is a parent of a reticulation $v_a$ and $v_b$, respectively, then no tree in $\cP$ has $\{a, b\}$ as a cherry if $v_a=v_b$, and at most a quarter of the trees in $\cP$ have $\{a, b\}$ as a cherry if $v_a\neq v_b$. Thus we may assume that $p_a$ is an ancestor of $p_b$.  Since $\cN$ is normal, it is now easily checked that $(p_a, p_b)$ is an edge of $\cN$; otherwise, less than half of the trees in $\cP$ have $\{a, b\}$ as a cherry.
If the child of $p_b$ that is not $b$, say $w$, is a tree vertex or a leaf, then no tree displayed by $\cN$ has $\{a, b\}$ as a cherry, so $w$ is a reticulation.
Let $p_w$ be the parent of $w$ that is not $p_b$, and let $\ell$ and $m$ be leaves at the end of tree paths starting at $w$ and $p_w$, respectively. Note that $\ell\neq m$. Furthermore, as $\cN$ is normal, $(p_w, w)$ is not a shortcut, so $a\neq m$.

Since the display set of $\cN$ is $\cP$, there is a tree $\cT$ in $\cP$ whose unique embedding in $\cN$ uses the reticulation edge $(p_b, w)$. In particular, $\cT|\{a, b, \ell, m\}$ is isomorphic to $(\ell, b, a, m)$. By the initial hypothesis, this implies that there is a tree $\cT_1$ in $\cP$ such that $\cT_1|\{a, b, \ell, m\}$ is isomorphic to $(a, b, \ell, m)$. But $\cT_1$ is not displayed by $\cN$, a contradiction. Therefore $p_a$ is not an ancestor of $p_b$. A similar argument also shows that $p_b$ is not an ancestor of $p_a$. Hence either $p_a$ or $p_b$ is a reticulation.

If $p_b$ is a reticulation, then, as exactly half of the trees in $\cP$ have $\{a, b\}$ as a cherry, it is easily seen that $(p_a, p_b)$ is an edge in $\cN$. Thus $\{a, b\}$ is a reticulated cherry of $\cN$ with reticulation leaf $b$. Similarly, if $p_a$ is a reticulation of $\cN$, then $\{a,b\}$ is a reticulated cherry of $\cN$ with reticulation leaf $a$. We next show $\{a,b\}$ is not a reticulated cherry of $\cN$ with reticulation leaf $b$, thereby showing that $p_b$ is not a reticulation.

Suppose that $\{a,b\}$ is a reticulated cherry of $\cN$ with reticulation leaf $b$. Let $\ell$ be the leaf at the end of a tree path in $\cN$ starting at $g_b$. By considering the edge $(g_b, p_b)$, half of the trees in $\cP$ have $(b,\ell,a)$ as a triple (note that $p_a$, and therefore $a$, cannot be a descendant of $g_b$ else $(g_b,p_b)$ is a shortcut). Moreover, by the initial hypothesis and as $x=a$, if $\cT'$ is a tree in $\cP$ and $(b,\ell,a)$ is not a  triple of $\cT'$, then $\{a, b\}$ is a cherry of $\cT'$ and the cluster of $\cT'$ corresponding to the grandparent of $b$ in $\cT'$, excluding $a$, is a non-empty subset of the cluster of $g_b$ in $\cN$. Since $\cT'$ is displayed by $\cN$ and the unique embedding of $\cT'$ in $\cN$ uses $(p_a, p_b)$, it follows that there exists a directed path $P$ from the parent of $p_a$ to $g_b$ in $\cN$. 

We now show that $P$ consists of a single edge. Let $u$ be the parent of $p_a$, so $u$ is the first vertex in $P$, and let $v$ denote the second vertex in $P$. Assume that $v\neq g_b$. Say $v$ is a tree vertex and let $w$ be the child of $v$ that does not lie on $P$. Since $\cN$ has no shortcuts, such a vertex exists. Now there is a tree path from $w$ to a leaf $m$ regardless of whether $w$ is a leaf, a tree vertex, or a reticulation. Note that if $w$ is a leaf, then $m=w$. As $m$ is at the end of a tree path starting at $w$, it follows that $m$ is not a descendant of $g_b$ in $\cN$. Let $\cT$ be a phylogenetic $X$-tree that is displayed by $\cN$  and whose embedding uses the edges $(p_a, p_b)$ and $(v, w)$. Clearly $\cT$ has cherry $\{a,b\}$. Importantly, the cluster of the grandparent of $b$ in $\cT$ contains $m$, a contradiction. Therefore, $v$ is a reticulation. Let $v'$ be the parent of $v$ that is not $u$, and let $m'$ be a leaf at the end of tree path starting at $v'$. Since $(v', v)$ is not a shortcut, $m'\neq a$. Furthermore, $m'$ is also not a descendant of $g_b$ in $\cN$. But then again there is a tree displayed by $\cN$ whose unique embedding uses the edges $(p_a, p_b)$, $(v', v)$, and all edges on the subpath of $P$ from $v$ to $g_b$, and whose cluster of the grandparent of $b$ contains $m'$, another contradiction.  Hence $\cP$ consists of a single edge and the parent of $p_a$ is the parent of $g_b$.

Now let $\cT$ be a tree displayed by $\cN$ and whose unique embedding uses the edge $(g_b, p_b)$. Note that $(b,\ell,a)$ is a  triple of $\cT$, and the parents of $a$ and $b$ in $\cT$ are joined by an edge. By the initial hypothesis and since $x=a$, there is a tree $\cT_1$ in $\cP_1$ with $\varphi(\cT_1)=\cT$ such that the unique binary normal network, $\cM$ say, whose display set is $\{\cT_1, \cT\}$ has $\{a, b\}$ as a reticulated cherry with reticulation leaf $a$. As $\cM$ is normal, the leaf $\ell$ which is at the end of a tree path starting at $g_b$ in $\cN$ is not a descendant of $g_a$ in $\cM$, otherwise, $(a,\ell,b)$ is a  triple of $\cT$. Since $(b,\ell,a)$ is a  triple of $\cT$, the tree displayed by $\cM$ that uses $(g_a, p_a)$ does not have the property that the parents of $a$ and $b$ are joined by an edge. This last contradiction establishes that $\{a, b\}$ is not a reticulated cherry of $\cN$ with reticulation leaf $b$. It now follows that $\{a, b\}$ is  a reticulated cherry of $\cN$ with reticulation leaf $a$.

To complete the proof of (i), observe that the network obtained from $\cN$ by deleting the reticulation edge $(p_b,p_a)$ of the reticulated cherry $\{a, b\}$ with reticulation leaf $a$ and suppressing the resulting degree-two vertices is binary normal and has display set $\cP-\cP_1$. The proof of (ii) is an immediate consequence of the initial hypothesis.
\end{proof}

The next theorem shows that {\sc Display Set Compatibility} correctly decides if a collection of binary phylogenetic trees is tightly normal compatible.

\begin{theorem}
Let $\cP$ be a collection of binary phylogenetic $X$-trees. Then {\sc Display Set Compatibility} applied to $\cP$ correctly determines if $\cP$ is tightly normal compatible, in which case, it returns the unique binary normal network whose display set is $\cP$.
\label{dsc}
\end{theorem}

\begin{proof}
We may assume that $|\cP|=2^k$ for some non-negative integer $k$; otherwise, the theorem holds. The proof of the correctness of {\sc Display Set Compatibility} is by induction on $k+|X|$. If $k=0$ or $|X|\le 2$, the theorem clearly holds. Thus we may assume that $k\ge 1$ and $|X|\ge 3$. Suppose that the theorem holds for all collections of  binary phylogenetic $X'$-trees of size $2^{k'}$, where $k'$ is a non-negative integer and $k'+|X'| < k+|X|$.

First assume that $\cP$ is tightly normal compatible. Then there is a unique binary normal network $\cN$ on $X$ whose display set is $\cP$. If $\{x, y\}$ is a cherry of $\cN$, then, by Lemma~\ref{cherries}(i), all trees in $\cP$ have $\{x, y\}$ as a cherry. On the other hand  if $\cN$ has no cherries, then, by \cite[Lemma~4.1]{bor16}, $\cN$ has a reticulated cherry, say $\{x, y\}$ with reticulation leaf $x$. Furthermore, as the display set of $\cN$ is $\cP$, Lemma~\ref{cherries}(ii) implies that exactly half of the trees in $\cP$ have $\{x, y\}$ as a cherry and there is a bijection from the subset $\cP_1$ of trees in $\cP$ that have $\{x, y\}$ as a cherry to $\cP-\cP_1$ such that, for all $\cT_1$ in $\cP_1$, the $2$-elements subset $\{x, y\}$ of $X$ is a reticulated cherry with reticulation leaf $x$ of the unique binary normal network whose display set is $\{\cT_1, \varphi(\cT_1)\}$. Thus {\sc Display Set Compatibility} applied to $\cP$ constructs either (i) a collection $\cP'$ of binary phylogenetic $(X-\{a\})$-trees by replacing each tree $\cS$ in $\cP$ with $\cS\backslash a$ for some cherry $\{a, b\}$ of $\cN$ or (ii) a collection $\cP-\cP_1$ for some reticulated cherry $\{a, b\}$ with reticulation leaf $a$, where $\cP_1$ is the subset of trees in $\cP$ with $\{a, b\}$ as a cherry. If (i) holds, it follows by induction and Lemma~\ref{all-cherries} that {\sc Display Set Compatibility} applied to $\cP'$ correctly determines that $\cP'$ is tightly normal compatible and returns the unique binary normal network on $X-\{a\}$ whose display set is $\cP'$. In turn this implies that {\sc Display Set Compatibility} correctly determines that $\cP$ is tightly normal compatible and returns the unique  normal network on $X$ whose display set is $\cP$, namely, $\cN$. So assume that (ii) holds. Then, by induction and Lemma~\ref{half-cherries}(i), {\sc Display Set Compatibility} applied to $\cP-\cP_1$ correctly determines that $\cP-\cP_1$ is tightly normal compatible and returns the unique binary normal network on $X$, say $\cN'$, whose display set is $\cP-\cP_1$. The algorithm {\sc Display Set Compatibility} correctly reconstructs $\cN$ from $\cN'$. Therefore, as $\cN$ is normal, {\sc Display Set Compatibility} correctly determines that $\cP$ is tightly normal compatible and returns the unique binary normal network on $X$ whose display set is $\cP$, namely $\cN$.

Now assume that $\cP$ is not tightly normal compatible. If there is no pair of leaves $a$ and $b$ satisfying the conditions in Step~\ref{all-step} or Step~\ref{half-step}, then, by Lemma~\ref{cherries}, $\cP$ is not tightly normal compatible and {\sc Display Set Compatibility} applied to $\cP$ correctly determines this outcome. So assume that there is such a pair of leaves. Then {\sc Display Set Compatibility} applied to $\cP$ constructs either (I) a collection $\cP'$ of phylogenetic $(X-\{a\})$-trees from $\cP$ by replacing each tree $\cS$ in $\cP$ with $\cS\backslash a$ if $\{a, b\}$ is a cherry of every tree in $\cP$, or (II) a collection $\cP-\cP_1$ of $2^{k-1}$ binary phylogenetic $X$-trees, where $\cP_1$ is the subset of trees in $\cP$ in which $\{a, b\}$ is a cherry. If (I) holds, then, by Lemma~\ref{all-cherries}, $\cP'$ is not tightly normal compatible and so, by induction, {\sc Display Set Compatibility} applied to $\cP'$ correctly determines that $\cP'$ is not tightly normal compatible. In turn, this implies that {\sc Display Set Compatibility} applied to $\cP$ correctly determines that $\cP$ is not tightly normal compatible. On the other hand, if (II) holds, then, by the working of the algorithm, we may assume that there is a bijection $\varphi: \cP_1\rightarrow \cP-\cP_1$ such that, for all $\cT_1\in \cP$, the pair $\{a, b\}$ is a reticulated cherry with reticulation leaf $a$ of the unique binary normal network whose display set is $\{\cT_1, \varphi(\cT_1)\}$. If $\cP-\cP_1$ is not tightly normal compatible, then, by induction, {\sc Display Set Compatibility} applied to $\cP-\cP_1$ correctly determines that $\cP-\cP_1$ is not tightly normal compatible. It follows that {\sc Display Set Compatibility} applied to $\cP$ correctly determines that $\cP$ is not tightly normal compatible. If $\cP-\cP_1$ is tightly normal compatible, then, by induction, {\sc Display Set Compatibility} applied to $\cP-\cP_1$ returns the unique binary normal network $\cN'$ on $X$ whose display set is $\cP-\cP_1$. Let $\cN$ be the network on $X$ obtained from $\cN'$ by subdividing the edges directed into $a$ and $b$ with new vertices $p_a$ and $p_b$, respectively, and adding the new edge $(p_b, p_a)$. If $\cN$ is normal, then, as $a$ and $b$ satisfy the conditions in Step~\ref{half-step}, and the display set of $\cN'$ is $\cP-\cP_1$, it follows from Lemma~\ref{half-cherries}(ii) that the display set of $\cN$ is $\cP$. But then $\cP$ is tightly normal compatible, a contradiction. So $\cN$ is not normal, and it follows that {\sc Display Set Compatibility} applied to $\cP$ correctly determines that $\cP$ is not tightly normal compatible. This completes the proof of the theorem.
\end{proof}

We end the section with a discussion of the running time of {\sc Display Set Compatibility}. The most time consuming part of {\sc Display Set Compatibility} is deciding whether the conditions in Step~\ref{all-step} or Step~\ref{half-step} hold for some $2$-element subset of $X$. Naively, checking if Step~\ref{all-step} applies takes time $O(|X|^2|\cP|)$. Moreover, checking if half of the trees of $\cP$ have a common cherry takes time $O(|X|^2|\cP|)$. The next lemma shows that it takes $O(|X|^3|\cP|^2)$ time to decide if the bijection in Step~\ref{half-step} exists.

\begin{lemma}
To decide if the bijection in Step~\ref{half-step} exists takes $O(|X|^3 |\cP|^2)$ time.
\end{lemma}

\begin{proof}
If there is a bijection $\varphi$ for a $2$-element subset $\{a, b\}$ with $x=a$, then $\cT_1$ and $\varphi(\cT_1)$ have the property that $\cT_1\backslash a\cong \varphi(\cT_1)\backslash a$. Say there is a tree $\cT_2$ in $\cP_1$ such that $\cT_2\backslash a\cong \varphi(\cT_1)\backslash a$. Then $\cT_1\backslash a\cong \cT_2\backslash a$ and so, as $\{a, b\}$ is a cherry of $\cT_1$ and $\cT_2$, it follows that $\cT_1\cong \cT_2$. Hence if the bijection in Step~\ref{half-step} exists for $\{a, b\}$ with $x=a$, then it is unique, and so we simply need to find, for each $\cT_1\in \cP_1$, a tree $\cT'_1$ in $\cP-\cP_1$ such that $\cT'_1\backslash a\cong \cT_1\backslash a$ and check that there exists a binary normal network whose display set is $\{\cT_1, \cT'_1\}$ with a reticulated cherry $\{a, b\}$ in which $a$ is a reticulation leaf. For this check, we only need to decide if $b$ is not a descendant of the parent of $a$ in $\cT'_1$. In total, this takes time $O(|X||\cP|^2)$. Since the number of $2$-element subsets of $X$ is $O(|X|^2)$, deciding if the bijection in Step~\ref{half-step} exists takes $O(|X|^3|\cP|^2)$ time.
\end{proof}

\noindent Continuing the discussion, if $|\cP|=2^k$ for some non-negative integer $k$, then we only need to complete Step~\ref{half-step} at most $k$ times, that is, at most $O(|X|)$ times. To see this, note that Step~\ref{half-step} is invoked only if there is no cherry common to all the trees in $\cP$ and, if $\cP$ is tightly normal compatible and $\cN$ is the unique binary normal network whose display set is $\cP$, then $\cN$ has exactly $k$ reticulations. Lastly, {\sc Display Set Compatibility} is called at most $O(|X|+|X|)=O(|X|)$ times. This occurs when $\cP$ is tightly normal compatible or $\cP$ is the display set of a tree-child network with precisely one shortcut and this shortcut is a reticulation edge of a reticulated cherry. Hence, the total running time of {\sc Display Set Compatibility} is $O(|X|^4|\cP|^2)$.

\section{Two phylogenetic trees are always displayed by a normal network}\label{sec:two-normal}

In this section, we show that any two phylogenetic trees are normal compatible. In particular, we show that any two binary phylogenetic trees can be displayed by a binary normal network. An analogous two-tree result for tree-child networks has been known since 2005~\cite{baroni05} and an explicit proof is given in~\cite{linz19}. Moreover, we observe at the end of this section that there are three binary phylogenetic trees that are not normal compatible. This last result is in contrast to tree-child networks, where any set of phylogenetic trees can be displayed by a  tree-child network~\cite{linz19}, and also in contrast to  temporal tree-child networks, where there exist two binary phylogenetic trees that cannot be displayed by a temporal tree-child network~\cite{humphries13b}.
 
We start with a lemma that follows from~\cite[Theorem 1.1]{semple16}.

\begin{lemma}\label{l:embedding}
Let $\cN$ be a binary tree-child network on $X$, and let $\cE$ be a subset of the edges of $\cN$. Then $\cE$ is an embedding of a binary phylogenetic $X$-tree displayed by $\cN$ if and only if $\cE$ contains every tree edge of $\cN$ and, for each reticulation $v$ in $\cN$ with parents $u$ and $u'$, $\cE$ contains exactly one of $(u, v)$ and $(u',v)$.
\end{lemma}

\noindent Let $\cN$ be a binary tree-child network on $X$, and let $\cE$ be an embedding of a phylogenetic $X$-tree that is displayed by $\cN$. Let $R$ be the subset of $\cE$ that contains precisely each edge of $\cE$ that is a reticulation edge in $\cN$. Since $\cE$ contains each tree edge of $\cN$ by Lemma~\ref{l:embedding}, it follows that $R$ contains enough information to determine $\cE$. In what follows, we freely use this fact and describe an embedding by the set of reticulation edges in $\cN$ that it contains instead of all edges. We say that $R$ {\it induces} $\cE$ if $\cE$ is the union of $R$ and all tree edges in $\cN$. 

\begin{theorem}\label{t:two}
Let $\cT$ and $\cT'$ be two binary phylogenetic $X$-trees.  Then $\cT$ and $\cT'$  are normal compatible. 
\end{theorem}

\begin{proof}
Towards a contradiction, assume that there exist two binary phylogenetic $X$-trees $\cT$ and $\cT'$ that are not normal compatible. By~\cite[Corollary 1.4]{linz19}, there exists a binary tree-child network on $X$ that displays $\cT$ and $\cT'$.  Let $\cN$ be a binary tree-child network that displays $\cT$ and $\cT'$ with the minimum number of shortcuts amongst binary tree-child networks on $X$ that display $\cT$ and $\cT'$ and whose length of a shortest shortcut is minimised over all binary tree-child networks on $X$ that display $\cT$ and $\cT'$ with the minimum number of shortcuts. 

Let $k$ be the number of shortcuts in $\cN$, and let $l$ be the minimum length of a shortcut in $\cN$. Furthermore, let $\cE$ and $\cE'$ be embeddings of $\cT$ and $\cT'$, respectively, in $\cN$. Throughout the remainder of the proof, we freely assume that each reticulation edge of $\cN$ is used by exactly one of $\cE$ or $\cE'$ because, if that is not the case, then there exists a tree-child network on $X$ that displays $\cT$ and $\cT'$ with less than $h(\cN)$ reticulations, at most $k$ shortcuts, and whose length of a minimum shortcut is at most $l$. Let $v$ be a reticulation in $\cN$ with parents $u$ and $u'$ such that $(u,v)$ is a shortcut of length $l$. Without loss of generality, we may assume that $\cE$ uses $(u,v)$ and $\cE'$ uses $(u',v)$. Let $R$ (resp. $R'$) be the subset of the reticulation edges in $\cN$ that are used by $\cE$ (resp. $\cE'$).

In the following, we construct a tree-child network $\cN'$ on $X$ from $\cN$ that displays $\cT$ and $\cT'$ and has either less than $k$ shortcuts, or $k$ shortcuts one of which with length less than $l$, thereby deriving a contradiction in both cases. If $l=2$, then $(u,u')$ is an edge in $\cN$. It follows that $(R-\{(u,v)\})\cup\{(u',v)\}$ induces an embedding of $\cT$ in $\cN$ that does not use $(u,v)$. Hence, the binary tree-child network $\cN'$ obtained from $\cN$ by deleting $(u,v)$ and suppressing $u$ and $v$ displays $\cT$ and $\cT'$ and has $k-1$ shortcuts, a contradiction. We therefore assume that $l>2$. Let $u=u_1, u_2,\ldots, u_m=u'$ be a directed path $P$ from $u$ to $u'$. Since $(u,v)$ is a shortcut of minimum length, each tree vertex on $P$ has a child that does not lie on $P$. Let $i$ be the maximum element in $\{1,2,\ldots,m-1\}$ such that  $u_i$ is a reticulation, $u_i$ is a tree vertex whose child that does not lie on $P$ is a tree vertex or a leaf, or $u_i=u_1$. Observe that $u_i$ is a vertex of each directed path from $u$ to $u'$ and, for each $j\in\{i+1,i+2,\ldots,m-1\}$, the child of $u_j$ that does not lie on $P$ is a reticulation,  $w_j$ say. The setup is illustrated in Figure~\ref{fig:two-tree-proof}(i). Assume that there exists a $j\in\{i+1,i+2,\ldots,m-1\}$ such that $(u_j,w_j)$ is used by $\cE$. Choose $j$ so that no reticulation edge in $\{(u_{j+1},w_{j+1}),(u_{j+2},w_{j+2}),\ldots,(u_{m-1},w_{m-1})\}$ is used by $\cE$. Let $q$ be the child of $u'$ that is not $v$. Now obtain $\cN'$ from $\cN$ by deleting $(u_j,w_j)$, suppressing $u_j$, subdividing $(u',q)$ with a new vertex $u_j'$, and adding the edge $(u_j',w_j)$. Figure~\ref{fig:two-tree-proof}(i)--(ii) shows the construction of $\cN'$ from $\cN$. It is straightforward to check that, as $\cN$ is a binary tree-child network, $\cN'$ is also such a network. Let $u_j''$ be the parent of $w_j$ in $\cN$ (resp. $\cN'$) that is not $u_j$ (resp. $u_j'$). If $(u''_j,w_j)$ is a shortcut in $\cN'$, then $u_j''$ is an ancestor of $u_j'$ and, in turn, an ancestor of $u_i$. It follows that  $(u_j'',w_j)$ is a shortcut in $\cN$. On the other hand, if  $(u_j',w_j)$ is a shortcut in $\cN'$, then $u_j''$ is a descendant of $q$, thereby implying that $(u_j,w_j)$ is also a shortcut in $\cN$. Thus, $\cN'$ has at most $k$ shortcuts. To see that $\cN'$ displays $\cT$ and $\cT'$, observe that $R'$ induces an embedding of $\cT'$ in $\cN'$ and, because no reticulation edge in $\{(u_{j+1},w_{j+1}),(u_{j+2},w_{j+2}),\ldots,(u_{m-1},w_{m-1})\}$ in an edge of $R$, it follows that
$$(R-\{(u_j,w_j)\})\cup\{(u_j',w_j)\}$$
induces an embedding of $\cT$ in $\cN'$. As $u_j$ does not lie on a directed path from $u$ to $u'$ in $\cN'$, $(u,v)$ has length $l-1$ in $\cN'$, we derive a contradiction. We continue with the proof under the following assumption.

\noindent {\bf (A)} For each  $j\in\{i+1,i+2,\ldots,m-1\}$, the edge $(u_j,w_j)$ is used by $\cE'$. 

\begin{figure}[h!]
\center
\scalebox{0.92}{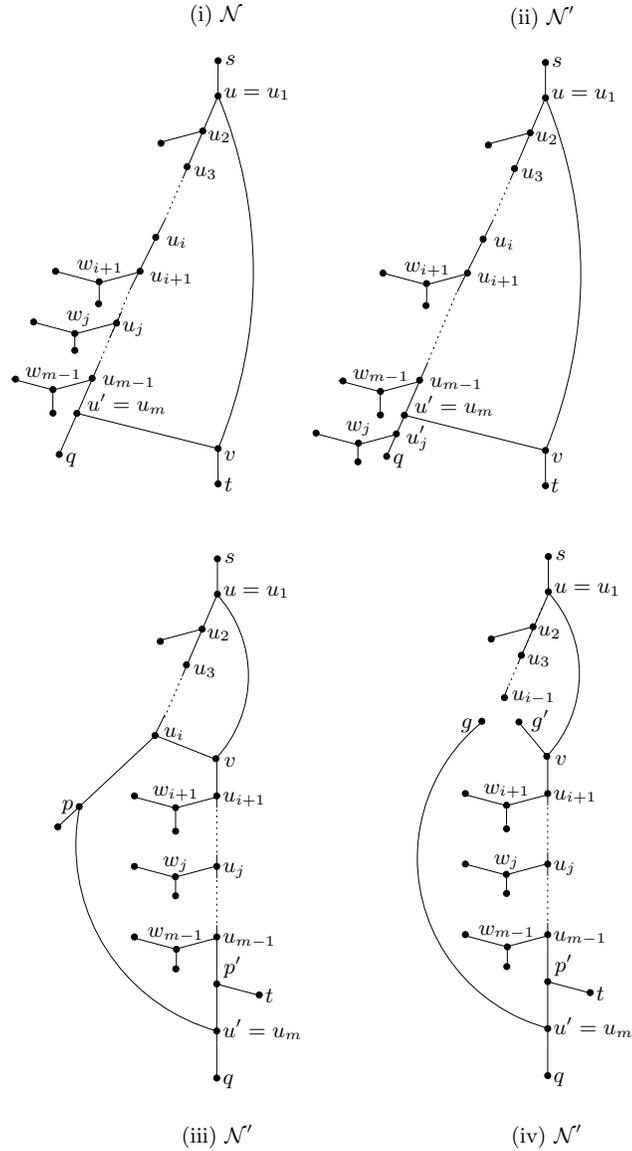}
\caption{(i) The tree-child network $\cN$ with a shortcut $(u,v)$ of length $l$ with $l\geq m$ as described in the proof of Theorem~\ref{t:two}. The tree-child network $\cN'$ obtained from $\cN$ with a shortcut $(u,v)$ of length strictly less than $l$ for when (ii) there exists a reticulation edge $(u_j,w_j)$ with $j\in\{i+1,i+2,\ldots,m-1\}$ that is used by $\cE$ and no such reticulation edge exists and $u_i$ is a tree vertex (iii) or a reticulation (iv). In (iv), one of $g$ or $g'$ is $u_{i-1}$. Some parts of $\cN$ and $\cN'$ are omitted.}
\label{fig:two-tree-proof}
\end{figure}

\noindent Let $s$ be the parent of $u$ in $\cN$, and let $t$ be the child of $v$ in $\cN$. We next consider three cases. 

First, assume that $u_i=u_1$. Then, for each $j\in\{2,3,\ldots,u_{m-1}\}$, the edge $(u_j,w_j)$ is only used by $\cE'$.  Similar to the construction in the last paragraph, obtain $\cN'$ from $\cN$ by deleting $(u_2,w_2)$, suppressing $u_2$, subdividing $(s,u)$ with a new vertex $u_2'$, and adding the edge $(u_2',w_2)$. As $\cN$ is a binary tree-child network with $k$ shortcuts, $\cN'$ is such a network as well because $w_2$ is incident with a shortcut  in $\cN'$ precisely if $w_2$ is  incident with a shortcut in $\cN$. Moreover, $R$ induces an embedding of $\cT$ in $\cN'$ and  $$(R'-\{(u_2,w_2)\})\cup\{(u_2',w_2)\}$$ induces an embedding of $\cT'$ in $\cN'$. Lastly, since $u_2$ does not lie on any directed path from $u$ to $u'$ in $\cN'$, the length of $(u,v)$ is $l-1$ in $\cN'$, another contradiction.

Second, assume that $u_i\ne u_1$ and that $u_i$ is a tree vertex. We next obtain a binary phylogenetic network $\cN'$ on $X$ from $\cN$ by applying the following edge deletions, subdivisions, and additions. Intuitively, these operations turn $u_i$ into a parent of $v$ and move the subpath of $P$ from $u_{i+1}$ to $u_m$, which contains at least $u_m$, below $v$, thereby turning $u_m$ into an additional reticulation and shortening the shortcut $(u,v)$. More precisely, obtain $\cN'$ from $\cN$ by 
\begin{enumerate}[(i)]
\item deleting the edges $(u_m,v)$, $(u_i,u_{i+1})$, and $(v,t)$,
\item subdividing the tree edge that is directed out of $u_i$ and not incident with $u_{i+1}$ with a new vertex $p$, and the edge  $(u_{m-1}, u_m)$ with a new vertex $p'$, and
\item adding the edges $(u_i,v)$, $(p,u_m)$, $(v,u_{i+1})$, and $(p',t)$.
\end{enumerate}

\noindent The construction of $\cN'$ from $\cN$ is shown in Figure~\ref{fig:two-tree-proof}(i) and (iii). Let $\ell$ be a leaf at the end of a tree path that starts at $v$ in $\cN$. Since the child of $v$ in $\cN'$ is either $u_{i+1}$ or $p'$, which are both tree vertices, and each vertex in $\{u_{i+1},u_{i+2},\ldots,u_{m-1},p'\}$ has a tree path that starts at that vertex and ends at $\ell$ in $\cN'$, it is straightforward to check that, as $\cN$ is tree-child, $\cN'$ is also tree-child. Moreover, $$R\cup\{(p,u_m)\} \text{ and } (R'-\{(u_m,v)\})\cup\{(u_i,v),(p',u_m)\}$$ induce an embedding of $\cT$ and $\cT'$, respectively in $\cN'$. Lastly, we turn to shortcuts in $\cN'$. It follows from the construction that $u_m$ is a reticulation in $\cN'$ with parents $p$ and $p'$. Since $p$ is neither an ancestor nor a descendant of $p'$, $u_m$ is not incident with a shortcut in $\cN'$. For each $j\in \{i+1,i+2,\ldots,m-1\}$,  let $u_j'$ be the parent of $w_j$ that is not $u_j$ in $\cN$. Observe that $u_j$ and $u_j'$ are also the parents of $w_j$ in $\cN$. Furthermore, by Assumption (A), it follows that $u_j'\notin\{u_{i+1},u_{i+2},\ldots u_m\}$. Hence, if $(u_{j},w_j)$ (resp. $(u_j',w_j)$) is a shortcut in $\cN'$, then $(u_{j},w_j)$ (resp. $(u_j',w_j)$) is a shortcut in $\cN$. Thus $\cN'$ has at most $k$ shortcuts and, importantly, the shortcut $(u,v)$ has length at most $l-1$ because $u_1$ and $u_i$ are the two parents of $v$ in $\cN'$ and  $u_m$ does not lie on a directed path from $u_1$ to $u_i$, another contradiction.

Third assume that $u_i\ne u_1$ and $u_i$ is a reticulation.  Let $g$ and $g'$ be the two parents of $u_i$. At least one of $g$ and $g'$ lies on $P$. Without loss of generality, we may assume that $\cE$ uses $(g,u_i)$ and  $\cE'$ uses $(g',u_i)$. We next obtain a binary phylogenetic network $\cN'$ on $X$ from $\cN$ in which $u_m$ is a reticulation, $u_i$ is not a vertex, $g$ is a parent of $u_m$, and $g'$ is a parent of $v$.  More precisely, in order, obtain $\cN'$ from $\cN$ by \begin{enumerate}[(i)]
\item deleting the edges $(u_m,v)$ and $(v,t)$,  
\item deleting the vertex $u_i$,
\item adding the edges $(g',v)$, $(g,u_m)$, and $(v,u_{i+1})$,
\item subdividing the edge that is directed into $u_m$ and not incident with $g$ with a new vertex $p'$, and
\item  adding the edge $(p',t)$.
\end{enumerate}

\noindent The construction of $\cN'$ from $\cN$ is shown in Figure~\ref{fig:two-tree-proof}(i) and (iv). As $\cN$ is tree-child, note that $g$ and $g'$ are tree vertices in $\cN$ and $\cN'$. We can now apply the same argument as in the second case to show that $\cN'$ is tree-child. Moreover 
$$(R-\{(g,u_i)\})\cup\{(g,u_m)\}\text{ and }(R'-\{(g',u_i),(u_m,v)\})\cup\{(g',v),(p',u_m)\}$$
induce an embedding of $\cT$ and $\cT'$, respectively, in $\cN'$. Turning to shortcuts in $\cN'$, assume that $(g,u_m)$ is a shortcut. Then $g$ is an ancestor of $p'$ in $\cN'$. By construction and Assumption (A), this in turn implies that $g$ is an ancestor of $g'$ in $\cN'$ and $\cN$. Thus $(g,u_i)$ is a shortcut in $\cN$. Now assume that $(p',u_m)$ is  a shortcut in $\cN'$. Then $g$ is a descendant of $t$ in $\cN'$ and in $\cN$. But $t$ is also a descendant of $g$ in $\cN$ because there is a directed path from $u_i$, a child of $g$, to $v$. Thus, $\cN$ contains a directed cycle, thereby implying that  $(p',u_m)$ is not a shortcut in $\cN'$. It  follows that, if $u_m$ is  incident with a shortcut in $\cN'$, then $u_i$ is incident with a shortcut in $\cN$. Again, for each $j\in\{i+1,i+2,\ldots,m-1\}$, let $u_j'$ be the parent of $w_j$ that \sout{does not lie on $P$} is not $u_j$ in $\cN$. Applying the same argument as in the second case, we deduce that, if $(u_
{j},w_j)$ (resp. $(u_j',w_j)$) is a shortcut in $\cN'$, then $(u_{j},w_j)$ (resp. $(u_j',w_j)$) is a shortcut in $\cN$. It remains to show that $v$ is either incident with no shortcut or with a shortcut of length at most $l-1$ in $\cN'$. Assume that $g'$ does not lie on any directed path from $u$ to $u'$ in $\cN$. Then there is no directed path from $u$ to $g'$ in $\cN'$, thereby implying that $(u,v)$ is not a shortcut in $\cN'$. Hence, if $(u,v)$ is a shortcut in $\cN'$, then $g'$ and $u_i$ both lie on a directed path from $u$ to $u'$ in $\cN$, whereas $u_i$ does not lie on a directed path from $u$ to $g'$ in $\cN'$. Hence, $(u,v)$ has length at most $l-1$. In summary, it  follows that $\cN'$ has at most $k$ shortcuts. Moreover, if $\cN'$ has $k$ shortcuts, then $(u,v)$ is one of them and has length at most $l-1$, a final contradiction. The theorem now follows.
\end{proof}

Observe that the process of obtaining $\cN'$ from $\cN$ as described in the proof of Theorem~\ref{t:two} can be repeatedly applied to reduce the length of a shortcut to two, at which point it can be removed. For two binary phylogenetic trees $X$-trees $\cT$ and $\cT'$, this gives an algorithmic approach to construct a binary normal network that displays $\cT$ and $\cT'$. More specifically, start with a binary tree-child network $\cN$ that displays $\cT$ and $\cT'$ and then keep applying the process of removing or shortening a shortcut to each shortcut until the resulting network becomes normal.

The next corollary is an immediate consequence of Theorem~\ref{t:two} because, if  a binary refinement of a phylogenetic tree is displayed by a phylogenetic network, then the (unrefined) tree is also displayed by the same network.

\begin{corollary}\label{c:two}
Let $\cT$ and $\cT'$ be two  phylogenetic $X$-trees. Then $\cT$ and $\cT'$ are normal compatible.
\end{corollary}

Since each binary normal network on $X$ has at most $|X|-2$ reticulations~\cite{bickner12,mcdiarmid15} and a display set of size $2^k$, where $k$ is the number of reticulations in $\cN$ and each such vertex has in-degree two, it immediately follows that Theorem~\ref{t:two} does not generalise to more than two phylogenetic $X$-trees. To see this, consider the three binary phylogenetic trees with leaf set $\{a,b,c\}$. Each phylogenetic network $\cN$ that displays these trees with a single reticulation $v$ has the property that the in-degree of $v$ is three and, regardless of which leaf is the child of $v$, it is straightforward to check that $\cN$ has a shortcut. Moreover, there exists no normal network on three leaves with two reticulations. Hence, the three binary phylogenetic trees with leaf set $\{a,b,c\}$ are not normal compatible.

\section{Sets of phylogenetic trees displayed by a normal network}\label{sec:multi-normal}

In this section, we characterise $h_{\rm n}(\cP)$ for a collection $\cP$ of phylogenetic $X$-trees in terms of a particular type of sequence on the elements in $X$ called a ``normal cherry-picking sequence''. Let $\cT$ be a phylogenetic $X$-tree and let $(x, y)$ be an ordered pair of leaves in $X$. If $(x, y)$ is a cherry of $\cT$, then let $\cT'=\cT\backslash x$; otherwise, let $\cT'=\cT$. We say that $\cT'$ has been obtained from $\cT$ by {\em cherry picking $(x, y)$}. 

Let $\sigma=(x_1, y_1), (x_2, y_2), \ldots, (x_s, y_s)$ be a sequence of ordered pairs in $X\times X$. We refer to $\sigma$ as a {\it shortcut sequence} if there are indices $i_1<i_2<\cdots<i_m$ with $i_1=1$ and $i_m=s$ such that either $m=2$, $x_{i_1}=x_{i_2},$ and $y_{i_1}=y_{i_2}$, or $m>2$ and the following four properties are satisfied:

\begin{enumerate}
\item $x_{i_1}=x_{i_m}$,
\item for each $j\in\{2,3,\ldots,m-1\}$, we have $x_{i_1}\notin\{x_{i_j},y_{i_j}\}$, 
\item the elements $x_{i_2}, x_{i_3},\ldots,x_{i_{m-1}}$ are distinct, and 
\item for each $j\in\{2, 3, \ldots, m\}$, the intersection $\{x_{i_{j-1}},y_{i_{j-1}}\}\cap\{x_{i_{j}},y_{i_{j}}\}$ is non-empty. 
\end{enumerate}

\noindent If $\sigma$ is a shortcut sequence, then we refer to $(x_{i_1},y_{i_1}),(x_{i_2},y_{i_2}),\ldots,(x_{i_m},y_{i_m})$ as the subsequence of $\sigma$ that {\it verifies} $\sigma$.

Now, let $\cP$ be a set of phylogenetic $X$-trees, and let
$$\sigma=(x_1, y_1), (x_2, y_2), \ldots, (x_s, y_s), (x_{s+1}, -)$$
be a sequence of ordered pairs in $X\times (X\cup \{-\})$ such that the following  property is satisfied.

\noindent {\bf (TC)} For all $i\in \{1, 2, \ldots, s\}$, we have $x_i\not\in \{y_{i+1}, y_{i+2}, \ldots, y_s\}$.

\noindent  Setting $\cP_0=\cP$ and, for all $i\in \{1, 2, \ldots, s\}$, setting $\cP_i$ to be the set of phylogenetic trees obtained from $\cP_{i-1}$ by cherry picking $(x_i, y_i)$ in each tree in $\cP_{i-1}$, we call $\sigma$ a {\em tree-child cherry-picking sequence} for $\cP$ if each tree in $\cP_s$ consists of the single vertex $x_{s+1}$. Furthermore, for all $i\in \{1, 2, \ldots,s\}$, we say that $\cP_i$ is obtained from $\cP$ by {\it picking $x_1, x_2, \ldots, x_{i}$}.  Additionally, if $\cP_i\ne\cP_{i+1}$, then we refer to $(x_i,y_i)$ as being {\it essential.}  Let $\sigma$ be a tree-child cherry-picking sequence for $\cP$, then the {\it weight} of $\sigma$, denoted $w(\sigma)$, is the value $s+1-|X|$. Observe that, 
$$w(\sigma)=s+1-|X|\ge 0$$
as each element in $X$ must appear as the first element in an ordered pair in $\sigma$.

Let $\sigma=(x_1, y_1), (x_2, y_2), \ldots, (x_s, y_s), (x_{s+1}, -)$ be a tree-child cherry-picking sequence for a collection $\cP$ of phylogenetic $X$-trees. For two elements $i,j\in\{1,2,\ldots, s+1\}$, we denote by $\sigma[i,j]$ the substring of $\sigma$ that starts at the $i$th ordered pair and ends at the $j$th ordered pair.  We say that $\sigma$ is a {\it normal cherry-picking sequence} for $\cP$ if it additionally satisfies the following property.

\noindent {\bf (N)} For each pair $(x_i,y_i)$ and $(x_j,y_j)$ with $i<j$, the substring $\sigma[i, j]$ is not a shortcut sequence. 

\noindent To illustrate these notions,  
$$\sigma=(\ell_1,\ell_2),(\ell_2,\ell_3),(\ell_3,\ell_4),(\ell_1,\ell_4),(\ell_4,-)\text{ and }$$ $$\sigma'=(\ell_2,\ell_1),(\ell_2,\ell_3),(\ell_3,\ell_4),(\ell_3,\ell_1),(\ell_1,\ell_4),(\ell_4,-)$$ are two tree-child cherry-picking sequences for the two caterpillars $\cT$ and $\cT'$ with $n=4$ that are shown in Figure~\ref{fig:caterpillars}. Since $\sigma[1, 4]$ is a shortcut sequence, $\sigma$ is not a normal cherry-picking sequence whereas $\sigma'$ is such a sequence. 
Most sequences of ordered pairs of elements in $X\times (X\cup \{-\})$ that we consider in this paper are normal cherry-picking sequences and the more general  tree-child cherry-picking sequences are only needed for technical reasons. We refer to (TC) as the {\it tree-child property} and to (N) as the {\it normal property}. Observe that checking if an arbitrary sequence of ordered pairs in $X\times (X\cup \{-\})$ satisfies these properties does not require any information about the elements in $\cP$. 

Now, let $\sigma$ be a normal cherry-picking sequence for a collection  $\cP$ of phylogenetic $X$-trees. We call $\sigma$ a {\it minimum normal cherry-picking sequence} of $\cP$ if $w(\sigma)$ is minimised over all normal cherry-picking sequences of $\cP$. This smallest value is denoted by $s(\cP)$. 

\begin{lemma}\label{l:three-properties}
Let $\cP$ be a set of  phylogenetic $X$-trees. Let $\sigma$ be a tree-child cherry-picking sequence for $\cP$. Then there exists a tree-child network $\cN$ on $X$ that displays $\cP$ with $h(\cN)\le w(\sigma)$ and satisfies the following four properties:
\begin{enumerate}[{\rm (i)}]
\item If $u$ is a tree vertex in $\cN$ and not a parent of a reticulation, then there are leaves $\ell_1$ and $\ell_1'$ at the end of tree paths starting at the two children  of $u$, such that $(\ell_1, \ell_1')$ is an element in $\sigma$.
           
\item If $u$ is a tree vertex in $\cN$ that has a child that is a reticulation $v$, then there are leaves $\ell_u$ and $\ell_v$ at the end of tree paths starting at $u$ and $v$, respectively, such that $(\ell_v, \ell_u)$ is an element in $\sigma$.

\item If $u$ and $u'$ are tree vertices and $v$ a reticulation of $\cN$ such that $v$ is a child of $u$ and $u'$, and $u'$ is the parent of $u$, then there are leaves $\ell_1$ and $\ell_1'$  at the end of tree paths starting at the two children of $u$ such that $(\ell_1,\ell_1'),(\ell_1,\ell_1')$ is a subsequence of $\sigma$.

\item Let $u$ and $u'$ be tree vertices of $\cN$ such that either $(u,v)$ and $(v,u')$ are the edges of a path of length two and $v$ is a reticulation, or $(u,u')$ is an edge and no reticulation of $\cN$ is a child of $u$ and $u'$. If $\ell_1$ and $\ell_1'$ are leaves  at the end of tree paths starting at the two children of $u'$ such that $(\ell_1,\ell_1')$ is an ordered pair in $\sigma$, then there exist leaves $\ell_2$ and $\ell_2'$ at the end of tree paths starting at the two children of $u$ such that   $(\ell_1,\ell_1'),(\ell_2,\ell_2')$ is a subsequence of $\sigma$, $\{\ell_1,\ell_1'\}\cap\{\ell_2,\ell_2'\}\ne\emptyset$ and $\ell_1\ne\ell_2$.
\end{enumerate}
\end{lemma}

\begin{proof}
Properties (i) and (ii) are established in~\cite[Lemma 3.1]{linz19}.  Moreover, Property (iii) follows from applying Property (ii) twice, and Property (iv) follows from Properties (i) and (ii), and a straightforward extension of the proof of~\cite[Lemma 3.1]{linz19}.
\end{proof}

\begin{lemma}\label{l:one-binary}
Let $\cP$ be a set of  binary phylogenetic $X$-trees. Suppose that there exists a normal cherry-picking sequence $\sigma$ for $\cP$. Then there exists a normal network $\cN$ on $X$ that displays $\cP$ with $h(\cN)\le w(\sigma)$.
\end{lemma}

\begin{proof}
Without loss of generality, we may assume that each ordered pair in $\sigma$ is essential. It follows from Lemma~\ref{l:three-properties} that there exists a tree-child network $\cN$ on $X$ that satisfies Properties (i)--(iv) as stated in the same lemma and displays $\cP$ with $h(\cN)\le w(\sigma)$. We  complete the proof by showing that $\cN$ is also normal. Towards a contradiction, assume that $\cN$ has a shortcut. Let $(u',v)$ be a shortcut of minimum length over all shortcuts in $\cN$. Furthermore, let $u$ be a parent of $v$ that is not $u'$ such that there exists a directed path $P$ from $u'$ to $u$ in $\cN$. If the second child of $u'$ that is not $v$ equates to $u$, then, as $\cN$ satisfies Property (iii) in the statement of Lemma~\ref{l:three-properties}, there are  leaves $\ell_1$ and $\ell_1'$ at the end of tree paths starting at the two children of $u$ such that $(\ell_1,\ell_1'),(\ell_1,\ell_1')$ is a subsequence of $\sigma$. Hence $\sigma$ contains a shortcut sequence, a contradiction.  We may therefore assume for the remainder of the proof that the length of $(u',v)$ is at least three. In order, let $u'=u_m,u_{m-1},\ldots,u_2,u_1=u$ be the tree vertices on $P$, where $m\geq 3$ as $\cN$ is tree-child. Furthermore,  let $w_1$ denote the child of $u_1$ that is not $v$ and, for each $j\in\{2, 3, \ldots, m-1\}$, let  $w_j$  be the child of $u_j$ in $\cN$ that does not lie on $P$. As $(u',v)$ is a shortcut of minimum length, there is no reticulation on $P$ that has two parents that also both lie on $P$ and, thus, $w_j$ exists for each $j\in\{2, 3, \ldots, m-1\}$. Moreover, because $\cN$ is tree-child, observe that the child of $u_j$ that is not $w_j$ is either  $u_{j-1}$ or  a reticulation in which case the child of  that reticulation  equates to $u_{j-1}$. The setup is shown in Figure~\ref{fig:multi-tree-proof}.

\begin{figure}[t]
\center
\scalebox{0.92}{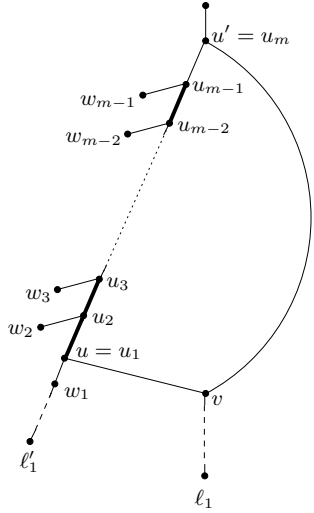}
\caption{The setup of $\cN$ as described in the proof of Lemma~\ref{l:one-binary}. Each of the three thick lines indicates directed paths that consists of at most two edges and each of the two dashed lines indicates a tree path.}
\label{fig:multi-tree-proof}
\end{figure}

As $\cN$ satisfies Property (ii) in the statement of Lemma~\ref{l:three-properties}, there exist a leaf $\ell_1$  at the end of a tree path starting at $v$ and a leaf $\ell_1'$ at the end of a tree path starting at $u_1$ such that $(\ell_1,\ell_1')$ is an element in $\sigma$. Now consider each $j\in\{2,3,\ldots,m-1\}$ in order.  Since $\cN$  satisfies Property (iv) in the statement of Lemma~\ref{l:three-properties}, there exist leaves $\ell_j$ and $\ell_j'$  at the end of tree paths starting at the two children of $u_j$ and leaves $\ell_{j-1}$ and $\ell'_{j-1}$ at the end of tree paths starting at the two children of $u_{j-1}$ such that $(\ell_{j-1},\ell_{j-1}'),(\ell_j,\ell_j')$ is a subsequence of $\sigma$, $\{\ell_{j-1},\ell_{j-1}'\}\cap \{\ell_j,\ell_j'\}\ne\emptyset$, and $\ell_{j-1}\ne\ell_j$. Lastly, consider $u_m$. By Properties (ii) and (iv) in the statement of Lemma~\ref{l:three-properties}, there exist a leaf $\ell_m$  at the end of a tree path starting at $v$ and a leaf  $\ell_m'$ at the end of a tree path starting at $u_{m-1}$ such that $(\ell_{m-1},\ell_{m-1}'),(\ell_m,\ell_m')$ is a subsequence of $\sigma$, $\{\ell_{m-1},\ell_{m-1'}'\}\cap \{\ell_m,\ell_m'\}\ne\emptyset$, and $\ell_{m-1}\ne\ell_m$. It now follows that
$$(\ell_{1},\ell_{1}'),(\ell_2,\ell_2'),\ldots,(\ell_{m},\ell_{m}')$$
is a subsequence of $\sigma$ such that, for each $j\in\{2,3,\ldots,m\}$, the intersection $\{\ell_{j-1},\ell_{j-1}'\}\cap \{\ell_j,\ell_j'\}$ is non-empty. By the choice of $(u',v)$, $v$ is not a child of a vertex in $\{u_2,u_3,\ldots,u_{m-1}\}$ and, thus, $\ell_1$ is not a leaf at the end of a tree path that starts at a child of a vertex in $\{u_2,u_3,\ldots,u_{m-1}\}$. Hence, no ordered pair $(\ell_j,\ell_j')$ with $j\in\{2,3,\ldots,m-1\}$ contains $\ell_1$. 

Now assume that there exist elements $i,j\in\{2,3,\ldots,m-1\}$ with $i<j$ such that $\ell_i=\ell_{j}$. Without loss of generality, we can choose $i$ and $j$ such that there exists no $j'\in\{i+1,i+2,\ldots,j-1\}$ with $\ell_{j'}=\ell_{j}$. Since $\ell_j\neq \ell_{j+1}$ for all $j\in \{1, 2, \ldots, m-1\}$, we have $i+1<j$. Consider the subsequence
 $$(\ell_i,\ell_i'),\ldots,(\ell_{j-1},\ell_{j-1}'), (\ell_{j},\ell_{j}')$$
 of $\sigma$. Since $\ell_{i}=\ell_j$, it follows that $\ell_{j}'$ is a leaf at the end of a tree path starting at $w_{j}$. Moreover, by the choice of $j$ and because $\sigma$ satisfies (TC), we have $\ell_{j}\notin\{\ell_{j-1},\ell_{j-1}'\}$. As $\ell_j'$ is not a leaf at the end of any tree path starting at the two children of $u_{j-1}$, it now follows that $\{\ell_{j-1},\ell_{j-1}'\}\cap\{\ell_{j},\ell_{j}'\}=\emptyset$, thereby contradicting that $\cN$ satisfies Property (iv) in the statement of Lemma~\ref{l:three-properties}. We conclude that the elements $\ell_2,\ell_3,\ldots,\ell_{m-1}$ are distinct. 

We next show that $\ell_1=\ell_m$. As $\sigma$ satisfies (TC), observe that $\ell_m'\ne\ell_1$. Assume that $\ell_1\ne\ell_m$. Then $\ell_1$ and $\ell_m$ are two distinct leaves at the end of tree paths starting at $v$. Furthermore, each phylogenetic $X$-tree displayed by $\cN$ is binary and contains the triple $(\ell_1,\ell_m,\ell_1')$. Thus, each tree in $\cP$ also contains the triple $(\ell_1,\ell_m,\ell_1')$. As each ordered pair in $\sigma$ is essential, there exist binary phylogenetic $X$-trees $\cT$ and $\cT'$ in $\cP$ such that $(\ell_1,\ell_1')$ picks $\ell_1$ from $\cT$ and $(\ell_m,\ell_m')$ picks $\ell_m$ from $\cT'$. Clearly, $\cT$ and $\cT'$ are distinct. Since $(\ell_1,\ell_m,\ell_1')$ is a triple of $\cT$, there exists a subsequence $\sigma_1$ of $\sigma$ such that $\ell_m$ is the first coordinate of the first ordered pair of $\sigma_1$ and $\ell_1$ is the second coordinate of the last ordered pair of $\sigma_1$. Similarly, since $(\ell_1,\ell_m,\ell_1')$ is a triple of $\cT'$, there exists a subsequence $\sigma_2$ of $\sigma$ such that $\ell_1$ is the first coordinate of the first ordered pair of $\sigma_2$ and $\ell_m$ is the second coordinate of the last ordered pair of $\sigma_2$. Since $\sigma$ satisfies (TC), the last ordered pair of $\sigma_1$ precedes the first ordered pair of $\sigma_2$. This implies that the first ordered pair of $\sigma_1$ precedes the last ordered pair of $\sigma_2$ in $\sigma$. Hence, $\sigma$ does not satisfy (TC), a contradiction. We conclude that $\ell_1=\ell_m$. It now follows that $(\ell_1,\ell_1'),(\ell_2,\ell_2'),\ldots,(\ell_m,\ell_m')$ verifies a shortcut sequence of $\sigma$, thereby contradicting that $\sigma$ satisfies (N). This final contradiction completes the proof of the lemma.
\end{proof}

We now generalise Lemma~\ref{l:one-binary} to arbitrary collections of phylogenetic trees.

\begin{corollary}\label{c:one-non-binary}
Let $\cP$ be a set of  phylogenetic $X$-trees. Suppose that there exists a normal cherry-picking sequence $\sigma$ for $\cP$. Then there exists a normal network $\cN$ on $X$ that displays $\cP$ with $h(\cN)\le w(\sigma)$.
\end{corollary}

\begin{proof}
Let $\cP=\{\cT_1,\cT_2,\ldots,\cT_n\}$. For each $\cT_i\in\cP$ with $i\in\{1,2,\ldots,n\}$, let $\cT'_i$ be a binary refinement of $\cT_i$ such that $\sigma$ is a normal cherry-picking sequence for $\cT_i'$. It  follows from Lemma~\ref{l:one-binary} that there exists a normal network $\cN$ on $X$ that displays $\{\cT_1',\cT_2',\ldots,\cT_n'\}$ with $h(\cN)\le w(\sigma)$. By construction, $\cN$ also displays $\cP$ and, thus, the corollary follows.
\end{proof}

The next corollary follows from Corollary~\ref{c:one-non-binary} by choosing a normal cherry-picking sequence $\sigma$ of minimum weight for $\cP$.

\begin{corollary}\label{c:normal-one}
Let $\cP$ be a set of phylogenetic $X$-trees. Suppose that there exists a normal cherry-picking sequence for $\cP$. Then $h_{\rm n}(\cP)\leq s(\cP)$.
\end{corollary}

Let $\cN$ be a tree-child network on $X$ with root $\rho$ that displays a collection $\cP$ of phylogenetic $X$-trees, and let $v_1, v_2, \ldots, v_r$ denote the reticulations of $\cN$. Furthermore, let $\ell_{\rho}, \ell_1, \ell_2, \ldots, \ell_r$ denote the leaves at the end of tree paths starting at $\rho, v_1, v_2, \ldots, v_r$, respectively.  For each $j\in\{1,2,\ldots,r\}$, we say that $\ell_j$ {\it verifies} $v_j$. The following algorithm {\sc Construct Sequence}, which was first published in~\cite{linz19}, constructs a sequence $\sigma$ of ordered pairs in $X\times (X\cup \{-\})$ from $\cN$. It was shown in~\cite[Lemma 3.4]{linz19} that $\sigma$ is a tree-child cherry-picking sequence for $\cP$. After stating the algorithm, we show that $\sigma$ is in fact a normal cherry-picking sequence for when $\cN$ is normal.

\noindent{\sc Construct Sequence}
\begin{enumerate}[{\bf Step 1.}]
\item Set $\cN=\cN_0$ and $\sigma_0$ to be the empty sequence. Set $i=1$.

\item \label{step:loop} If $\cN_{i-1}$ consists of a single leaf $x_i$, then set $\sigma_i$ to be the concatenation of $\sigma_{i-1}$ and $(x_i, -)$, and return $\sigma_i$.

\item If $\{x_i, y_i\}$ is a cherry in $\cN_{i-1}$, then
\begin{enumerate}[{\bf (a)}]
\item \label{step:3a} If one of $x_i$ and $y_i$, say $x_i$, equates to $\ell_j$ for some $j\in \{1, 2, \ldots, r\}$ and $v_j$ is not a reticulation in $\cN_{i-1}$, then set $\sigma_i$ to be the concatenation of $\sigma_{i-1}$ and $(x_i, y_i)$.

\item Otherwise, set $\sigma_i$ to be the concatenation of $\sigma_{i-1}$ and $(x_i, y_i)$, where $x_i\not\in\ \{\ell_\rho, \ell_1, \ell_2, \ldots, \ell_r\}$.

\item Set $\cN_i$ to be the tree-child network obtained from $\cN_{i-1}$ by deleting $x_i$.

\item Increase $i$ by one and go to Step~\ref{step:loop}.
\end{enumerate}

\item Else, there is a reticulated cherry $\{x_i, y_i\}$ in $\cN_{i-1}$, where $x_i$ say is the reticulation leaf.
\begin{enumerate}[{\bf (a)}]
\item \label{step:4a}Set $\sigma_i$ to be the concatenation of $\sigma_{i-1}$ and $(x_i, y_i)$.

\item Set $\cN_i$ to be the tree-child network obtained from $\cN_{i-1}$ by deleting $(p_{y_i},p_{x_i})$ and suppressing the two resulting degree-two vertices.

\item Increase $i$ by one and go to Step~\ref{step:loop}.
\end{enumerate}
\end{enumerate}

\begin{lemma}\label{l:normal-two}
Let $\cP$ be a set of phylogenetic $X$-trees. Suppose that there exists a normal network $\cN$ on $X$ that displays $\cP$. Then there exists a normal cherry-picking sequence $\sigma$ for $\cP$ with $w(\sigma)\leq h(\cN)$.
\end{lemma}

\begin{proof}
Let $\sigma=(x_1,y_1),(x_2,y_2),\ldots,(x_{s},y_{s}),(x_{s+1},-)$ be the sequence of ordered pairs in $X\times (X\cup \{-\})$  that is obtained from applying {\sc Construct Sequence} to $\cN$. It immediately follows from~\cite[Lemma 3.4]{linz19} that $\sigma$ is a tree-child cherry-picking sequence  for $\cP$ with $ w(\sigma)\leq h(\cN)$. By inspecting {\sc Construct Sequence}, we also have the following observation: If two ordered pairs of $\sigma$ have the same first coordinate, say $\ell$, then there exists a reticulation $v$ in $\cN$ such that $\ell$ is a leaf at the end of a tree path starting at $v$. Lastly, in each iteration $i$ of {\sc Construct Sequence} with $i\in\{1,2,\ldots,s\}$, it follows from straightforward generalisations of~\cite[Lemmas 6 and 7]{doecker20} to normal networks whose reticulations have in-degree at least two that $\cN_i$ is normal. 

To complete the proof, we show that $\sigma$ also satisfies (N). Towards a contradiction, assume that $\sigma$ contains a substring $\sigma[i,i']$ that is a shortcut sequence for some $i,i'\in\{1,2,\ldots,s\}$ with $i<i'$. Let $\sigma_s=(\ell_1,\ell_1'),(\ell_2,\ell_2'),\ldots,(\ell_m,\ell_m')$ be the subsequence of $\sigma[i,i']$ that verifies the shortcut sequence. By definition of a shortcut sequence, we have $x_i=\ell_1=\ell_m=x_{i'}$. As $\cN$ does not contain a shortcut, it also follows that $\ell_1'\ne \ell_m'$ and, thus, $m\geq 3$. Hence by the observation in the last paragraph, there exists a reticulation $v$ in $\cN$, such that $\ell_1$ is a leaf at the end of a tree path starting at $v$. Furthermore, $\{\ell_1,\ell_1'\}$ is a reticulated cherry  with reticulation leaf $\ell_1$ in  $\cN_{i-1}$, and $\{\ell_1,\ell_m'\}$ is either a cherry or a reticulated cherry with reticulation leaf $\ell_1$ in  $\cN_{i'-1}$. Let $p$ be the parent of $v$ in $\cN_{i-1}$ such that $p$ is a vertex of $\cN_{i'-1}$ but not a vertex of $\cN_{i'}$. Note that $p$ is a tree vertex of $\cN_{i'-1}$.

For each ordered pair $(x_j, y_j)$ in $\sigma$, either $x_j$ and $y_j$ have the same parent if $\{x_j, y_j\}$ is a cherry of $\cN_{j-1}$ or the parent of $y_j$ is a grandparent of $x_j$ if $\{x_j, y_j\}$ is a reticulated cherry with reticulation leaf $x_j$ of $\cN_{j-1}$. 
Let $(\ell_j, \ell'_j)$, $(\ell_{j'}, \ell'_{j'})$, and $(\ell_{j''}, \ell'_{j''})$ be three ordered pairs in $\sigma$ that are consecutive ordered pairs in $\sigma_s$. By Property (4) in the definition of a shortcut sequence and because $\sigma_s$ satisfies (TC), $\ell'_j\in \{\ell_{j'}, \ell'_{j'}\}$ and $\ell'_{j'}\in \{\ell_{j''}, \ell'_{j''}\}$. It now follows that the parent of $\ell'_{j''}$ is an ancestor of the parent of $\ell'_{j'}$ in $\cN_{j''-1}$ and, similarly, that the parent of $\ell'_{j'}$ is an ancestor of the parent of $\ell'_j$ in $\cN_{j'-1}$. Hence, the parent of $\ell_m'$ in $\cN_{i'-1}$, which is $p$, is an ancestor of $\ell'_1$  and therefore an ancestor of $\ell_1$ in $\cN$. But then $(p, v)$ is a shortcut in $\cN$, a contradiction as $\cN$ is normal. Thus $\sigma$ has no shortcut sequences, thereby completing the proof of the lemma.
\end{proof}

The next corollary follows from Lemma~\ref{l:normal-two} by choosing a normal network that displays a collection $\cP$ of phylogenetic trees  and whose hybridisation number is minimised over all such networks.

\begin{corollary}\label{c:normal-two}
Let $\cP$ be a set of phylogenetic $X$-trees. Suppose that $\cP$ is normal compatible. Then $h_{\rm n}(\cP)\geq s(\cP)$.
\end{corollary}

The following theorem summarises Lemmas~\ref{l:one-binary} and \ref{l:normal-two}, and Corollaries~\ref{c:normal-one} and~\ref{c:normal-two}.

\begin{theorem}\label{t:normal-characterisation}
Let $\cP$ be a set of phylogenetic $X$-trees. Then $\cP$ is normal compatible if and only if there exists a normal cherry-picking sequence for $\cP$, in which case $h_{\rm n}(\cP)=s(\cP)$.
\end{theorem}

\section{Disparity between tree-child and normal networks}\label{sec:disparity}

In this section, we use the results of Section~\ref{sec:multi-normal} to show that, for all $n\ge 3$, there exists a pair of binary phylogenetic $X$-trees $\cT_1$ and $\cT_2$ on $n$ leaves such that $h_{\rm tc}(\cT_1, \cT_2)=1$, but $h_{\rm n}(\cT_1, \cT_2)=n-2$. In particular, we establish the following proposition.

\begin{proposition}
Let $\cT_1$ be the caterpillar $(\ell_1, \ell_2, \ldots, \ell_n)$ and let $\cT_2$ be the caterpillar $(\ell_2, \ell_3, \ldots, \ell_n, \ell_1)$, where $n\ge 3$. Then $h_{\rm tc}(\cT_1, \cT_2)=1$ and $h_{\rm n}(\cT_1, \cT_2)=n-2$.
\end{proposition}

\begin{proof}
Since $\cT_1$ and $\cT_2$ are distinct, $h_{\rm tc}(\cT_1, \cT_2)\ge 1$. Let $\cN$ be the binary tree-child network obtained from $\cT_2$ by subdividing the edges directed into $\ell_2$ and $\ell_1$ with the new vertices $p_2$ and $p_1$ and adding the new edge $(p_2, p_1)$. Then $\cN$ displays $\cT_1$ and $\cT_2$, and has exactly one reticulation. Thus $h_{\rm tc}(\cT_1, \cT_2)=1$.

To show that $h_{\rm n}(\cT_1, \cT_2)=n-2$, let $\sigma$ be a normal cherry-picking sequence for $\cT_1$ and $\cT_2$. Without loss of generality, we may assume that every ordered pair in $\sigma$ is essential. Suppose that, for some $i\in \{2, 3, \ldots, n\}$, there is an ordered pair in $\sigma$ of the form $(\ell_1, \ell_i)$ that picks $\ell_1$ in $\cT_1$. If $\sigma$ picks $\ell_1$ in $\cT_2$ (strictly) before $\sigma$ picks $\ell_1$ in $\cT_1$, then $\sigma$ contains an ordered pair of the form $(\ell_1, \ell_{i'})$ corresponding to picking $\ell_1$ in $\cT_2$. But then, if $i\neq i'$, the normal cherry-picking sequence $\sigma$ has already picked $\ell_i$ in $\cT_2$ (otherwise $(\ell_1, \ell_{i'})$ is not yet a cherry), a contradiction as $\sigma$ satisfies (TC). In the case that $i=i'$, it follows that $(\ell_1, \ell_i)$ appears twice in $\sigma$, and so $\sigma$ is a shortcut sequence, another contradiction. If $\sigma$ picks $\ell_1$ in $\cT_1$ and $\cT_2$ simultaneously, then if $i\neq n$ we have $(\ell_n, \ell_i)$ as an ordered pair in $\sigma$ picking $\ell_n$ in $\cT_2$ and appearing before $(\ell_1, \ell_i)$, and $(\ell_n, \ell_i)$ as an ordered pair in $\sigma$ picking $\ell_n$ in $\cT_1$ and appearing after $(\ell_1, \ell_i)$. This verifies that $\sigma$ is a shortcut sequence.

It now follows that, unless $i=n$, we may assume that $\sigma$ picks $\ell_1$ in $\cT_2$ after it picks $\ell_1$ in $\cT_1$. Say there is an ordered pair in $\sigma$ of the form $(\ell_1, \ell_{i'})$ that picks $\ell_1$ in $\cT_2$. If $i=i'$, then $(\ell_1, \ell_i)$ appears twice in $\sigma$, a contradiction. If $i\neq i'$, then $\sigma$ picks $\ell_i$ before $\ell_{i'}$ in $\cT_2$, so the subsequence of $\sigma$ beginning with $(\ell_1, \ell_i)$ and ending with $(\ell_1, \ell_{i'})$ and whose intermediate ordered pairs consist of those pairs that pick leaves in $\cT_2$ starting with the ordered pair that picks $\ell_i$ in $\cT_2$ verifies that $\sigma$ is a shortcut sequence, a contradiction. Thus the ordered pair that picks $\ell_1$ in $\cT_2$ is of the form $(\ell_1, -)$.

We now deduce that either $(\ell_1, -)$ is the last ordered pair in $\sigma$ or, if this does not hold, then $(\ell_1, \ell_n)$ is the ordered pair in $\sigma$ that picks $\ell_1$ in $\cT_1$ and $\cT_2$ simultaneously. For the former to happen, $(\ell_1, -)$ picks $\ell_1$ in $\cT_1$ and $\cT_2$. For the latter to happen, the last two ordered pairs in $\sigma$ are $(\ell_1, \ell_n)$ and $(\ell_n, -)$. In both outcomes, $\sigma$ contains the subsequence
$$(\ell_2, \ell_1), (\ell_3, \ell_1), \ldots, (\ell_{n-1}, \ell_1)$$
corresponding to picking the leaves $\ell_2, \ell_3, \ldots, \ell_{n-1}$ in $\cT_1$. By comparison, the second coordinates of the ordered pairs in $\sigma$ that pick the leaves $\ell_2, \ell_3, \ldots, \ell_{n-1}$ in $\cT_2$ are distinct from those picking $\ell_2, \ell_3, \ldots, \ell_{n-1}$ in $\cT_1$. Hence $h_{\rm n}(\cT_1, \cT_2)\ge n-2$. But a binary normal network with $n$ leaves has at most $n-2$ reticulations~\cite{bickner12,mcdiarmid15}, and so $h_{\rm n}(\cT_1, \cT_2)=n-2$.
\end{proof}



\begin{thebibliography}{99}

\bibitem{baroni05}
M. Baroni, S. Gr\"unewald, V. Moulton, and C. Semple (2005). Bounding the number of hybridisation events for a consistent evolutionary history. Journal of Mathematical Biology, 51:171--182.

\bibitem{bernardi23}
G. Bernardini, L. van Iersel, E. Julien, and L. Stougie (2023). Reconstructing phylogenetic networks via cherry picking and machine learning. Algorithms for Molecular Biology, 18:13.

\bibitem{bickner12}
D. R. Bickner (2012). On normal networks. PhD thesis. Iowa State University.

\bibitem{blais21}
C. Blais and J. M. Archibald (2021). The past, present and future of the tree of life. Current Biology, 31:R311--R329.

\bibitem{bor16}
M. Bordewich and C. Semple (2016). Determining phylogenetic networks from inter-taxa distances. Journal of Mathematical Biology, 73:283--303.

\bibitem{cardona09}
G. Cardona, F. Rossell\'o, and G. Valiente (2009). Comparison of tree-child phylogenetic networks. IEEE/ACM Transactions on Computational Biology and Bioinformatics, 6:552--569.

\bibitem{doecker24}
J. D\"ocker, S. Linz, and C. Semple (2024). Hypercubes and Hamilton cycles of display sets of rooted phylogenetic networks. Advances in Applied Mathematics, 152:102595.

\bibitem{doecker20}
J. D\"ocker, S. Linz, and C. Semple (2020). Display sets of normal and tree-child networks. Electronic Journal of Combinatorics, 28, \#P1.8.

\bibitem{francis}
A. Francis (2021). ``Normal'' phylogenetic networks emerge as the leading class. arXiv:2107.10414.

\bibitem{francis21}
A. Francis, D. H. Huson, and M. Steel (2021). Normalising phylogenetic networks. Molecular Phylogenetics and Evolution, 136:107215. 

\bibitem{humphries13b}
P. J. Humphries, S. Linz, and C. Semple (2013). Cherry picking: A characterization of the temporal hybridization number for a set of phylogenies. Bulletin of Mathematical Biology, 75:1879--1890.

\bibitem{iersel22}
L. van Iersel, R. Janssen, M. Jones, Y. Murakami, and N. Zeh (2022). A practical fixed-parameter algorithm for constructing tree-child networks from multiple binary trees. Algorithmica, 84:917--960.
  
\bibitem{iersel10}
L. van Iersel, C. Semple, and M. Steel (2010). Locating a tree in a phylogenetic network. Information Processing Letters, 110:1037--1043.

\bibitem{janssen21}
R. Janssen and Y. Murakami (2021). On cherry-picking and network containment. Theoretical Computer Science, 856:121--150.

\bibitem{landry23}
K. Landry, A. Teodocio, M. Lafond, and O. Tremblay-Savard (2023). Defining phylogenetic network distances using cherry operations. IEEE/ACMTransactions on Computational Biology and Bioinformatics, 20:1654--1666.

\bibitem{landry}
K. Landry, O. Tremblay-Savard, and M. Lafond. A fixed-parameter tractable algorithm for finding agreement cherry-reduced subnetworks in level-$1$ orchard networks. Journal of Computational Biology, 31:360--379.

\bibitem{linz19}
S. Linz and C. Semple (2019). Attaching leaves and picking cherries to characterise the hybridisation number for a set of phylogenies. Advances in Applied Mathematics, 105:102--129.

\bibitem{mcdiarmid15}
C. McDiarmid, C. Semple, D. Welsh (2015). Counting phylogenetic networks. Annals of Combinatorics 19:205--224.

\bibitem{semple16}
C. Semple (2016). Phylogenetic networks with every embedded phylogenetic tree a base tree. Bulletin of Mathematical Biology, 78:132--137.

\bibitem{willson10}
S. J. Willson, S (2010). Properties of normal phylogenetic networks. Bulletin of  Mathematical Biology, 72:340--358.
  
\bibitem{willson11}
S. J. Willson (2011). Regular networks can be uniquely constructed from their trees. IEEE/ACM Transactions on Computational Biology and Bioinformatics, 8:785--796. 

\bibitem{willson12}
S. J. Willson (2012). Tree-average distances on certain phylogenetic networks have their weights uniquely determined. Algorithm for Molecular Biology 7:13.

\end{thebibliography}
\end{document}